\documentclass[12pt]{amsart}

\usepackage[latin1]{inputenc}

\usepackage{amsthm,amsmath,amssymb,amsfonts,enumerate,MnSymbol,latexsym,amsbsy}

\usepackage[usenames,dvipsnames,svgnames,table]{xcolor}

\usepackage[top=1in, bottom=1in, left=1in, right=1in]{geometry}

\usepackage[colorlinks=true, pdfstartview=FitV, linkcolor=ForestGreen,citecolor=ForestGreen, urlcolor=blue]{hyperref}

\usepackage{graphicx,caption}

\theoremstyle{plain}
\newtheorem{THEOREM}{Theorem}[section]

\newtheorem{theorem}[THEOREM]{Theorem}
\newtheorem{corollary}[THEOREM]{Corollary}
\newtheorem{lemma}[THEOREM]{Lemma}

\theoremstyle{definition}

\theoremstyle{remark}

\newtheorem{remark}[THEOREM]{Remark}

\newcommand{\thm}[1]{Theorem~\ref{#1}}
\newcommand{\lem}[1]{Lemma~\ref{#1}}

\def \a {\alpha}
\def \b {\beta}
\def \d {\delta}
\def \D {\Delta}
\def \g {\gamma}

\def \e {\varepsilon}

\def \l {\lambda}
\def \L {\Lambda}
\def \n {\nabla}
\def \s {\sigma}
\def \r {\rho}
\def \th {\theta}

\def \O {\Omega}

\def \bs {\boldsymbol{\sigma}}

\newcommand{\N}{\ensuremath{\mathbb{N}}}   
\newcommand{\Z}{\ensuremath{\mathbb{Z}}}   
\newcommand{\R}{\ensuremath{\mathbb{R}}}   
\newcommand{\T}{\ensuremath{\mathbb{T}}}   

\def \p {\partial}
\def \ra {\rightarrow}
\def \ss {\subset}
\def \bs {\backslash}

\def \per {\mathrm{per}}

\DeclareMathOperator{\diver}{div} %
\DeclareMathOperator{\symb}{Sym} %

\DeclareMathOperator{\Lip}{Lip} %

\DeclareMathOperator{\pv}{p.v.}

\begin{document}
\title{Global Well-Posedness of a Non-local Burgers Equation: the periodic case}

\author{Cyril Imbert} \thanks{The work of C. Imbert is partially
  supported by ANR grant ANR-12-BS01-0008-01. C.I. is grateful to the
  Department of Mathematics, Statistics, and Computer Science at the
  University of Illinois at Chicago for warm hospitality during his
  visit.}

\author{Roman Shvydkoy}\thanks{The work of R. Shvydkoy is partially
  supported by NSF grant DMS--1210896 and DMS--1515705. R.S. is grateful to the
  Laboratoire d'Analyse et de Math\'ematiques Appliqu\'ees at
  Universit\'e Paris-Est Cr\'eteil for warm hospitality during
  multiple visits.}

\author{Francois Vigneron} \thanks{F. Vigneron is grateful to the
  Department of Mathematics, Statistics, and Computer Science at the
  University of Illinois at Chicago for warm hospitality during multiple
  visits.}

\address[C. Imbert and F. Vigneron]{UPEC, LAMA, UMR 8050 du CNRS\\
	61 Avenue du G\'en\'eral de Gaulle\\
	94010 Creteil} 
\address[R. Shvydkoy]{Department of Mathematics, Stat. and Comp. Sci.\\
M/C 249\\
      University of Illinois\\
      Chicago, IL 60607}

\email{cyril.imbert@math.cnrs.fr}
\email{shvydkoy@uic.edu}
\email{francois.vigneron@u-pec.fr}

\begin{abstract}
  This paper is concerned with the study of a non-local Burgers
  equation for positive bounded periodic initial data. The equation reads
  \[
  u_t - u |\n| u + |\n|(u^2) = 0.
  \]
We construct global classical solutions starting from smooth positive data, and global weak solutions starting from data in $L^\infty$. 
We show that any weak solution is instantaneously regularized into $C^\infty$. We also describe the long-time behavior of all solutions.
Our methods follow several recent advances in the regularity theory of parabolic integro-differential equations.
\end{abstract}

\keywords{parabolic integro-differential equations, Burgers, Euler, surface quasi-geostrophic equation, Schauder estimates, maximum principle}

\subjclass[2010]{35K55, 45K05, 76B03}

\maketitle

\section{Introduction}

Dynamics of fluid motion provide a rich source of evolution laws that defy a complete well-posedness theory. Apart from the classical three dimensional Euler and Navier-Stokes equations, scalar models such as the super-critical surface quasi-geostrophic (SQG) equation or Darcy's law of porous media pose core difficulties, both to the traditional and newly developed approaches. In recent years, several classes of models have appeared that either mimic the basic structure of the aforementioned equations or serve as viable models on their own. For example, in relation to the SQG model
\[
\th_t + u\cdot \n \th =0
\]
where the div-free velocity $u = R^{\perp} \th$ is the perpendicular Riesz transform of $\th$, A. Cordoba, D. Cordoba, and M. Fontelos \cite{corcorfon} have studied an analogous 1D model given by
$\th_t + (H \th )\th_x = 0$,
where $H$ denotes the Hilbert transform. SQG written in divergence form $$\th_t + \diver(u \th)=0$$  received its counterpart in the form of $\th_t + (\th H\th)_x=0$ which was studied in \cite{cccf}, as well as convex combinations of the two 1D models above. The latter combination first appeared in physics literature \cite{bxm} as a model for a conservation law such as the classical Burgers equation $$u_t + \frac12(u^2)_x =0,$$ but with a non-local flux $F=\th H\th$. In these models the relation between the drift and the driven scalar is ensured by a Fourier multiplier with an odd symbol. This sets them apart from other laws such as porous media, Burgers or Euler equation or the recent Moffat's scalar model of magnetostrophic turbulence of the Earth's fluid core \cite{moffatt}. In all the later cases, such a relation is furnished via an even symbol which generally entails a more singular behavior.

\bigskip
In this paper we study the following model:
\begin{equation}\label{e:main}
u_t - u |\n| u + |\n|(u^2) = 0, \quad x\in \R^d \text{ or } \T^d,
\end{equation}
where $|\n|=(- \D)^{1/2}$ denotes the square root of the Laplacian and has the symbol $|\xi|$. With the addition of viscosity, the model $u_t + u |\n| u - |\n|(u^2) = \nu \D u$ was proposed by P.G. Lemari\'{e} as a scalar case study of the 3D Navier-Stokes (note the opposite signs). The works of F. Lelievre \cite{Lelievre-art1},  \cite{Lelievre-art2},  \cite{Lelievre-PHD}
presented the construction of global Kato-type mild solutions for initial data in $L^3(\R^3)$ and of global weak Leray-Hopf-type solutions for initial data in $L^2(\R^3)$, $L^2_{uloc}(\R^3)$ and $\dot{M}^{2,3}(\R^3)$. A local energy inequality obtained for this model was suggestive of possible uniqueness for small initial data in critical spaces, in a similar fashion to 3D Navier-Stokes.
The focus of the present paper will be on the inviscid case $\nu=0$.

Without viscosity the model bears resemblance to some of the ``even" inviscid cases in the sense explained above. For example, dropping the $1/2$ factor, the Burgers equation can be written in the form of a commutator: $$\partial_t u = [u,\partial_x]u.$$ Our model replaces $\p_x$ with the non-local operator $|\n|$ of the same order (hence our choice of name
for \eqref{e:main}).
The classical incompressible Euler equation is given by
\[
u_t + u \cdot \n u + \n p=0,
\]
where $p$ is the associated pressure given by $p = T(u \otimes u) + local$, where $T$ is a singular integral operator with an even symbol. We thus can draw an analogy between terms: $u \cdot \n u \sim -u|\n|u$ and $\n p \sim  |\n|(u^2)$. Actually, analogies with Euler or Burgers extend beyond the formal range. Let us discuss the basic structure of \eqref{e:main} in greater detail.  In what follows we give our
model~\eqref{e:main} the name of the \textsl{Non-local Burgers} equation.

\bigskip
At a formal level, \eqref{e:main} shares a more intimate connection with other equations of fluid mechanics. 
For example, if one applies formally the commutator theorem in the scalar case, even though $u$ might not be smooth and $|\xi|$ surely isn't, one gets:
$$[u,|\nabla|]
=\operatorname{Op} \left(\frac{i\xi}{|\xi|} \: u'(x) + \delta_{\xi=0}\: u''(x) + \ldots\right).$$
Thus in dimension 1, the equation becomes formally $\partial_t u = u' Ru + \left(\int_\R u\right) u'' +\ldots$
where $R$ denotes the Riesz transform.
The term $u'Ru$ is a scalar flavor of the SQG nonlinearity written above.

\bigskip
Let us recall that, in $\R^d$, the operator $|\n|$ enjoys an integral representation in terms of convolution with the kernel
$K(z) = \frac{c_d}{|z|^{d+1}}$ for some normalizing constant $c_d>0$ depending only on the dimension. The model  \eqref{e:main} can thus be rewritten in an integral form:
\begin{equation}\label{e:main2}
\partial_t u = \operatorname{p.v.} \int_{\R^d} K(y-x) (u(y)-u(x)) u(y)dy.
\end{equation}
If $u$ is periodic with period $2\pi$ in all coordinates, the representation above can alternatively be written as
\begin{equation}\label{e:main3}
\partial_t u = \operatorname{p.v.} \int_{\T^d} K_{\per}(y-x) (u(y)-u(x)) u(y)dy,
\end{equation}
where $\T^d$ is the torus and $K_{\per}(z)=\sum_{j \in \Z^d} \frac{c_d}{|z+2\pi j|^{d+1}}$. More explicitly  in 1D,  $K_{\per}(z) = \frac{1}{4\sin^2(z/2)}$. For periodic solutions, both representations are valid due to a sufficient decay of~$K$ at infinity. The former is more amenable to an analytical study due to the explicit nature of the kernel and applicability of known results, while the latter will be more useful in our 1D numerical simulations presented at the end.

\bigskip
The following basic structure properties of the model can be readily obtained from either representation, at a formal level. Let $u$ be a solution to~\eqref{e:main}.
\begin{itemize}
\item[(TI)] Translation invariance: if $x_0 \in \R^d$, $t_0 >0$ then $u(x+x_0,t+t_0)$ is another solution. In particular, the periodicity of the initial condition is preserved.
\item[(TR)] \label{TR} Time reversibility: if $t_0 >0$, then $-u(x,t_0 - t)$ is a solution too. 
\item[(SI)] Scaling invariance: for any $\l>0$ and $\a,\b\in \R$, $\lambda^\alpha u( \lambda^\beta x,\lambda^{\alpha +\beta} t)$ is another solution.
\item[(MP)] Max/Min principle: if $u>0$, then its maximum is decreasing and its minimum is increasing.
\item[(AMP)] Anti-Max/Min principle: if $u<0$, then its maximum is increasing and its minimum is decreasing.
\item[(E)] Energy conservation: $\|u(t)\|_{L^2} = \|u_0\|_{L^2}$ is obtained by testing \eqref{e:main} with $u$.
\item[(HP)] Higher power law: for any $p\in(2,\infty)$, the following quantity is conserved:
$$\|u(t)\|_{L^p}^p+ \frac{p}{2} \int_0^t \iint_{\Omega^2}  u(x)u(y) (|u(y)|^{p-2} -|u(x)|^{p-2}) (u(y) -u(x)) K(x,y)dxdy.$$
This equation, obtained by testing \eqref{e:main} with $|u|^{p-2}u$, implies the decay of those $L^p$ norms when $u>0$.
\item[(ML)] \label{ML} First momentum law: for either $\O = \R^d$ or $=\T^d$, integrating \eqref{e:main2} and using the Gagliardo-Sobolevskii representation of the $\dot{H}^{\frac12}(\O)$-norm (see \cite{hitchhiker}):
\begin{equation}\label{eq:ML}
\int_{\O} u(x,t') dx = \int_{\O} u(x,t) dx + \int_t^{t'} \|u(s)\|_{\dot{H}^{\frac12}(\O)}^2 ds.
\end{equation}
On $\T^d$, one has $L^2\subset L^1$ and this property combines nicely with (E) and ensures that $$u\in L^2(\R_+;\dot H^{1/2}(\T^d)),$$
regardless of the sign of $u_0$.
\end{itemize}
For positive solutions $u>0$ the right-hand side of \eqref{e:main2} gains a non-local elliptic structure of order $1$, while for $u<0$ the equation behaves like a backward heat equation. However, with the energy conservation (E), the model shares common features with conservative systems such as Euler, giving it a second nature. As a result, for $u>0$, we see an accumulation of energy in the large scales and a depletion of energy on small scales. More specifically, as seen from from the momentum law (ML), the dissipated energy from high frequencies gets sheltered in the first Fourier mode. At least on the phenomenological level, this property parallels what is known as the backward energy cascade in the Kraichnan theory of 2D turbulence (see \cite{kraich} and references therein). This makes our model potentially viable in studying scaling laws for the energy spectrum, structure functions, etc. 

\bigskip
The aim of this paper is to develop  a well-posedness theory for the model and study its long-time behavior. We make use of both sides of its dual nature by way of blending classical techniques relevant to the Euler equation, such as energy estimates, a Beale-Kato-Mayda criterion, etc \cite{mb-book}, with recently developed tools of regularity theory for  parabolic integro-differential equations \cite{ccv,cv,fk,iss,tianling,mp14}. Let us give a brief summary of our results with short references to the methods used.  Complete statements are given in Theorems~\ref{thm:local}, \ref{thm:BKM}, \ref{t:global}, \ref{t:weak}, \ref{thm:asympt}. First, we declare that all the results are proved in the periodic setting, except local existence which holds in both the periodic and the open case. Periodicity provides extra compactness of the underlying domain which for positive data, due to the minimum principle (MP), warrants  uniform support from below in space and time, which further entails uniform ellipticity of the right-hand side of \eqref{e:main2}. 

\medskip
\noindent {\bf Local existence.} For initial data $u_0 \in H^m(\O^d)$ on $\O^d = \R^d$ or $\T^d$, with $u_0>0$ pointwise and $m>d/2+1$, there exists a local solution in $C([0,T);H^m(\O^d)) \cap C^1([0,T); H^{m-1}(\O^d))$. Even for this local existence result, the positivity of the initial data is essential. We also have a BKM regularity criterion: if $\int_0^T \| \n u(t) \|_{L^\infty} dt<0$, the solution extends smoothly beyond $T$. The proof goes via a smoothing scheme based on a desingularization of the kernel.

\smallskip
\noindent {\bf Instant regularization.}  Any positive classical solution to \eqref{e:main} on a time interval $[0,T)$ satisfies the following bounds: for any
$k\in \N$, there exists $\a_k \in (0,1)$ such that for any
$0<t_0<T$:
\begin{equation}\label{e:instant}
\begin{split}
\| u \|_{C^{k+\a_k ,\a_k}_{x,t}(\T^d \times (t_0,T))} & \leq C (d,k, t_0,T,\min u_0, \max u_0)\\
\| \n_x^k u \|_{\Lip(\T^d \times (t_0,T))} & \leq C (d, k, t_0,T,\min u_0, \max u_0).
\end{split}
\end{equation}
To achieve this we symmetrize the right-hand side of \eqref{e:main2} by multiplying it by $u$ and then writing the evolution equation for $w = u^2$:
\begin{align}
 \partial_t w & = \pv \int (w(y)-w(x)) k(t,x,y) dy  \label{eq:w-lin} \\
 k(t,x,y) &= \frac{2 u(x,t) u(y,t)}{u(x,t)+u(y,t)} \frac{c_d}{|x-y|^{d+1}}. \label{eq:kernel}
\end{align}
The active kernel $k$ is symmetric and satisfies uniform ellipticity bound $\frac{\L^{-1}}{|z|^{d+1}}< k(z) < \frac{\L}{|z|^{d+1}}$. This puts the model within the range of recent results of Kassmann et al. \cite{bbrck,kass} and of Caffarelli-Chan-Vasseur \cite{ccv} where Moser / De Giorgi techniques were adopted to yield initial H\"older regularity for $w$ and hence for $u$ by positivity, i.e. estimate \eqref{e:instant} with the index $k=0$. See also \cite{komatsu,chen}. We then apply our new Schauder estimates for parabolic integro-differential equations with a general kernel \cite{iss}, see also \cite{tianling,mp14}, to obtain the full range of bounds \eqref{e:instant}.

\smallskip
\noindent {\bf Global existence.} It readily follows from the BKM criterion and the instant regularization property.

\smallskip
\noindent {\bf Global existence of weak solutions.} Since the bounds \eqref{e:instant} depend essentially only on the $L^\infty$-norm of the initial condition, we can construct global weak solutions starting from any $u_0\in L^\infty(\T^d)$, $u_0>0$. The solution belongs to the natural class $L^2([0,T); H^{1/2}) \cap L^\infty([0,T) \times \T^d)$ for all $T>0$. The initial data $u_0$ is realized both in the sense of
the $L^\infty $ weak$^*$-limit and in the strong topology of $L^2$. Such a solution satisfies \eqref{e:instant} instantly. A surprising difficulty emerged here in recovering the initial data since the sequence of approximating solutions from mollified data may not be weakly equicontinuous near time $t=0$. A weak formulation of \eqref{e:main} does not allow us to move the full derivative onto the test function. See Section~\ref{s:weak2} for a complete discussion.

\smallskip
\noindent {\bf Finite time blowup for $u_0<0$.} As follows from the previous discussion and (TR), a negative solution may start from $C^\infty$ data and develop into $L^\infty$ in finite time. We don't know whether a more severe instantaneous blowup occurs for negative data. Our requirement for local existence certainly supports this. However, the numerics presented in Section~\ref{par:numerics} suggests that for some mostly positive data with  a small subzero drop, the positive bulk of the solution may persevere. The solution gets dragged into positive territory and exists globally. 

\smallskip
\noindent {\bf Long-time asymptotics.} As $t \ra +\infty$, any weak solution to \eqref{e:main} converges to a constant, namely $\frac{\|u_0\|_{L^2}}{\sqrt{|\T^d|}}$, consistent with (E), in the following strong sense: the amplitude of $u(t)$ tends to $0$ exponentially fast with a rate proportional to $ \min u_0$. The semi-norm $\| \n u (t) \|_{L^\infty}$ does the same. Thus, there are no small-scale structures left in the limit. The latter statement requires more technical proof which relies on an adaptation of the recent Constantin-Vicol proof of regularity for the critical SQG equation, \cite{cv}.

\bigskip
The organization of the paper follows the order of the results listed above. To shorten the notations, we frequently use $\|\cdot\|_m$ to denote the Sobolev norm of $H^m$ and $| \cdot |_p$ to denote the $L^p$-norm.

\section{Global well-posedness with positive initial data}
\label{par:local}

\subsection{Local well-posedness}
We start our discussion with local well-posedness in regular classes. 
Let $\O^d$ denote $\R^d$ or $\T^d$. 
\begin{theorem}[Local well-posedness]\label{thm:local}
  Given a pointwise positive initial data $u_0 \in H^m(\O^d)$ where
  $m > d/2+1$ is an integer, there exists a time $T>0$ and a unique solution to
  \eqref{e:main} with initial condition $u_0$, which belongs to the
  class $C([0,T);H^m(\O^d)) \cap C^1([0,T); H^{m-1}(\O^d))$. Moreover, $u(x,t)>0$
  for all $(x,t) \in \O^d \times [0,T)$, and the maximum $\max_{\O^d}
  u(t)$ is strictly decreasing in time.
\end{theorem}
The proof in the case of $\R^d$ requires slightly more technical care in the maximum principle part, while being similar in the rest of the argument. We therefore present it in $\R^d$ only. In the case of the torus $\T^d$, however, we will also obtain a complementary statement for the minimum: $\min_{\T^d} u(t)$ is a strictly increasing function of time, thus the amplitude is shrinking. In Section~\ref{sec:ltb} we will elaborate much more on the asymptotic behavior of the amplitude. 
\begin{proof}
The proof will be split into several steps. 

\medskip
\noindent \emph{Step 1: Regularization.} Let us consider the following
regularization of the kernel
\[
K_\d(z) = \frac{c_d}{(\d^2+|z|^2)^{\frac{d+1}{2}}},
\]
and the corresponding operator
\[
|\n|_\d u = \int_{\R^d} K_\d(x-y) (u(x)-u(y)) dy = u \int_{\R^d} K_\d(z)dz -T_\d u
\]
where $T_\d$ is the convolution with $K_\d$.
The regularized equation takes the form
\begin{equation}
\partial_t u = [u,|\n|_\d]u =  \int_{\R^d} K_\d(x-y) (u(y)-u(x)) u(y)dy = -[u,T_\d]u.\label{regK}
\end{equation}
Note that $T_\d$ is infinitely smoothing, i.e.  $\|T_\d u\|_{s} \leq c_{\d,s,s'} \|
u\|_{s'}$, for any $0\leq s' \leq s$. So, by the standard quadratic
estimates, the right-hand side of \eqref{regK} is quadratically
bounded and locally Lipschitz in $H^m$. Thus, by the Fixed Point
Theorem, there is a local solution $u \in C^1([0,T); H^{m})$ with the
same initial condition $u_0$. Here $T$ depends on $\|u_0\|_m$ and
$\d$. For later use, note that $|u(t)|_2=|u_0|_2$ is conserved.

\medskip
\noindent \emph{Step 2: Maximum principle.} Suppose $u_0 \in H^m$,
$u_0 > 0$, and $u \in C^1([0,T); H^{m})$ is a local solution to
\eqref{regK} with the initial condition $u_0$. As $H^m(\R^d)
\hookrightarrow C^1(\R^d)$ for $m>d/2+1$ (even better is true), and
$u(x,t) \ra 0$ as $x\ra \infty$, then $u(t)$ has and attains its
maximum $M(t) = \max_{\R^d} u(t)$. We claim that $u(x,t)>0$, for all
$(x,t) \in \R^d \times [0,T)$, and the maximum function $M(t)$ is
strictly decreasing on $[0,T)$. Let us prove the first claim first.

Let us fix $R>0$ and show that $u$ never vanishes on $(0,T) \times
B_R(0)$. Suppose it does. Let us consider
\[
t_0 = \inf\{t\in (0,T): \exists |x|\leq R, u(x,t) = 0\}.
\]
By the boundedness of $(0,T) \times B_R(0)$ and the continuity of $u$, $t_0$
is attained. Thus, since $u_0>0$, then $t_0>0$. Let $x_0\in B_R(0)$ be
such that $u(x_0,t_0) = 0$. Evaluating \eqref{regK} at $(x_0,t_0)$ we
obtain
\[
u_t(x_0,t_0) =  \int_{\R^d} K_\d(x_0-y) u^2(y)dy >0.
\]
Observe that the right-hand side is strictly positive since the energy of solutions to \eqref{regK} is conserved.
This shows that for some earlier time $t<t_0$ there exists $x\in
B_R(0)$ where $u$ vanishes, which is a contradiction. Since the
argument holds for all $R>0$, the claim follows.

Let us prove the second claim now. Suppose that $M(t)$ is not strictly
decreasing on $[0,T)$. This implies that there exists a pair of times
$0\leq t' < t'' <T$ such that $M(t') \leq M(t'')$.
Let us show that there exists a $t_0 >t'$, such that $M(t_0) \geq M(t)$ for all $t\in[t', t_0]$.
If $M(t') < M(t'')$, then by the continuity of $M(t)$, $M$ attains its
maximum on the interval $[t',t'']$. Let $t_0\in[t',t''] $ be the left
outmost point where the maximum of $M$ is attained. Then $t_0 > t'$,
and $M(t_0) \geq M(t)$ for all $t'\leq t \leq t_0$.
If, on the contrary, $M(t')=M(t'')$ then either one can shrink the interval to fullfill the previous assumption or $M(t)$ is
constant throughout  $[t',t'']$. In either case, there exists, as claimed, a $t_0 >
t'$, such that $M(t_0) \geq M(t)$ for all $t'\leq t \leq
t_0$. Let us consider a point $x_0 \in \R^d$ such that $u(t_0,x_0) = M(t_0)$.
Then
\begin{equation}
\partial_t u(t_0,x_0) = \int_{\R^d} K_\d(x-y) (u(y)-u(x_0)) u(y)dy<0.
\end{equation}
So, at an earlier time $t<t_0$, one must have $u(t,x_0)> M(t_0)$ in contradiction with the initial assumption.

\medskip
\noindent \emph{Step 3: $\d$-independent bounds.} Let us observe the following representation formula that follows easily from $u(x)-u(y)=\int_0^1 \nabla u((1-\lambda)x+\lambda y)\cdot (x-y) d\lambda$~:
\[
|\n|_\d u = \int_0^1 R_{\l \d} (\n u )\: d\l,
\]
where $R_{\g}$ is the smoothed (vector) Riesz transform given by a convolution with the  kernel
$\Phi_\g(z)=\frac{z}{(\g^2 + |z|^2)^\frac{d+1}{2}}$. Let us notice that $\Phi_\g(z)= \frac{1}{\g^d} \Phi_1\left(\frac{z}{\g}\right)$, while $\Phi_1$ is the $z$-multiple of the Poisson kernel. Therefore, on the Fourier side, 
\[
\widehat{\Phi_\g}(\xi) = \frac{i\xi}{|\xi|} e^{-\g |\xi|}.
\]
Since these symbols are uniformly bounded, the family of operators $\{ R^i_\g\}_{i,\g} $ is uniformly bounded in $L^2$ and any Sobolev space $H^k$. Thus, we have uniform estimates
\begin{equation}\label{eq:riesz}
\| |\n|_\d u \|_k \leq C(d,k) \| \n u\|_k,
\end{equation}
for all $k\geq 0$.  Finally, note that, in any space dimension, the full symbol of $|\n|_\d$ is given by
\begin{equation}\label{symD}
\symb |\n|_\d = |\xi| \int_0^1 e^{-\l \d |\xi|} d\l = \frac{1}{\d}(1-e^{-\d |\xi|}).
\end{equation}

\medskip
Let $s$ be a multi-index of order $|s| = m$. Differentiating \eqref{regK}, we obtain
\[
\p_t \p^s u = \sum_{0\leq \a < s} \p^{s-\a} u |\n|_\d \p^\a u + u |\n|_\d \p^s u - 2 |\n|_\d(u \p^s u) - \sum_{0<\a< s} |\n|_\d(\p^\a u \p^{s-\a} u).
\]
Let us test with $\p^s u$. We have:
\begin{equation}\label{eq:work}
\begin{split}
\frac12 \p_t |\p^s u|_2^2 & =  \int \p^s u \sum_{0\leq \a < s} \p^{s-\a} u |\n|_\d \p^\a u + \int u \p^s u |\n|_\d \p^s u - 2 \int \p^s u |\n |_\d(u \p^s u) \\
& - \int \p^s u \sum_{0<\a< s} |\n|_\d(\p^\a u \p^{s-\a} u).
\end{split}
\end{equation}
The middle two terms of \eqref{eq:work} contain derivatives of order $m+1$.
By symmetry however, they add up to $ - \int u \p^s u |\n|_\d \p^s u$. By the positivity of $u$, a bound for this term follows from the elementary identity $-a(a-b)\leq -\frac{1}{2}(a^2-b^2)$:
\[\begin{split}
- \int u \p^s u |\n|_\d \p^s u & = - \iint u(x) \p^s u(x)(\p^su(x)-\p^su(y)) K_\d(x-y)dxdy\\
&\leq  -\frac{1}{2} \iint u(x) (\p^su(x)^2-\p^su(y)^2) K_\d(x-y)dxdy \\
&\quad = -\frac{1}{2}  \int u |\n|_\d (\p^s u)^2 =  -\frac{1}{2}  \int (|\n|_\d u)  (\p^s u)^2.
\end{split}
\]
One thus gets, using \eqref{eq:riesz} for the last step:
\[
- \int u \p^s u |\n|_\d \p^s u  -\frac{1}{2}  \int (|\n|_\d u)  (\p^s u)^2 \lesssim |  |\n|_\d u |_\infty \| u\|_m^2 \lesssim \|  |\n|_\d u \|_{m-1} \|u\|_m^2
\lesssim \| u\|_m^3
\]
where all the constants in the inequalities are independent of $\d$.

The rest of the expression \eqref{eq:work} is simpler to deal with as it does not contain any other derivatives of order $m+1$.
To estimate the first sum we use the
Gagliardo-Nirenberg inequalities:
\begin{equation}\label{GN}
| \p^i u |_{\frac{2r}{|i|}} \leq |u|_\infty^{1- \frac{|i|}{r}} \| u\|_r^{\frac{|i|}{r}}, \quad 0\leq |i| \leq r.
\end{equation}
So, for any $0\leq \a < s$, we have
\[
\begin{split}
\int | \p^s u \p^{s-\a} u |\nabla|_\d \p^\a u | & \leq |\p^s u|_2 |\p^{s-\a} u|_{\frac{2(m-1)}{m-|\a|-1}} | |\n|_\d \p^\a u |_{\frac{2(m-1)}{|\a|}} \\
&\leq \|u\|_m |\n u|_\infty^{1- \frac{m-|\a|-1}{m-1}} \|\n u\|_{m-1}^{\frac{m-|\a|-1}{m-1}}
| |\n|_\d u|_\infty^{1- \frac{|\a|}{m-1}} \||\n|_\d u\|_{m-1}^{\frac{|\a|}{m-1}}\\
& \lesssim \|u\|_m^3.
\end{split}
\]
Using \eqref{eq:riesz}, we also estimate each term of the last sum of \eqref{eq:work}. For $0<\alpha<s$, one has:
\begin{equation}\label{}
\begin{split}
\int \p^s u |\n|_\d(\p^\a u \p^{s-\a} u) \lesssim \int |\p^s u|_2 |\n(\p^\a u \p^{s-\a} u)|_2.
\end{split}
\end{equation}
We have $|\n(\p^\a u \p^{s-\a} u)|_2 \leq |\p^\a \n u \p^{s-\a} u|_2 +
|\p^\a u \p^{s-\a} \n u|_2 $. We estimate the first term exactly as
previously. For the second term we obtain, again using
Gagliardo-Nirenberg inequalities,
\[
 |\p^\a u \p^{s-\a} \n u|_2 \leq | \p^\a u|_{\frac{2(m-1)}{|\a|-1}} | \p^{s-\a} \n u |_{\frac{2(m-1)}{m-|\a|}} \leq \| u\|_m^2.
 \]

We thus have obtained a Riccati-type differential inequality $\p_t \|u\|_m^3\leq C\|u\|_m^3$ which boils down to
\[
\p_t\|u\|_m \leq C \|u\|_m^2,
\]
with a constant $C$ independent of $\d$. This shows that the solution can
be extended to a time of existence $T$ independent of $\d$ as
well. Namely, we have the bound
\[
\| u(t) \|_m \leq \frac{\|u_0\|_m}{1- Ct \|u_0\|_m},
\]
and so the critical time is $T^* = (C\|u_0\|_m)^{-1}$. 

\medskip
\noindent \emph{Step 4: Limit.} For each $\d>0$, let $u_\d$ be the
solution to \eqref{regK} with the same initial data $u_0$. By the
previous reasoning, $u_\d \in C([0,T); H^m)$ uniformly in $\d$ for any fixed
time $T<T^*$. Let us fix $T<T^*$. Then, since $H^{m-1}$ is a Banach
algebra, we estimate the right-hand side of \eqref{regK} by:
\[
\| u_\d |\n |_\d u_\d - |\n |_\d(u_\d^2) \|_{m-1} \leq \|u_\d\|_{m-1} \|u_\d\|_{m}+ \|u_\d\|_{m}^2 \lesssim \|u_\d\|_{m}^2.
\]
This shows that $u_\d \in C^1([0,T); H^{m-1})$ uniformly in $\d$. 

We now turn to the convergence issue. Instead of relying on
the Lions-Aubin compactness lemma, which only provides a limit for a
subsequence, we show a more robust convergence statement for the family
$u_\d$ as $\d\ra 0$. We claim that the family is a Cauchy sequence in $C([0,T);
L^2)$. As a consequence of the interpolation inequality $\| f\|_{m'}
\leq \|f\|_0^{1-m'/m} \|f\|_m^{m'/m}$, for $m'<m$, it also means
that the sequence is a Cauchy sequence in any $C([0,T); H^{m'})$ for any
$m'<m$. To prove our claim, we need another estimate on the difference of
operators $|\nabla|_\d $. Let us fix $\d,\e>0$. The symbol of the
difference is
\[
\symb (|\n|_\d - |\n|_\e) = |\xi| \int_0^1 ( e^{-\l \d|\xi|} - e^{-\l \e |\xi|}) d\l,
\]
which is bounded uniformly in  $\xi$ by $|\xi|^2 | \d- \e |$. Therefore we obtain the following bound:
\begin{equation}\label{e:diff}
\| (|\n|_\d - |\n|_\e) u \|_k \leq C |\d- \e | \| u\|_{k+2}.
\end{equation}
Writing the equation for the difference, we obtain
\[\begin{split}
(u_\e - u_\d)_t & = (u_\e - u_\d)|\n|_\e u_\e + u_\d(|\n|_\e - |\n|_\d) u_\e + u_\d|\n|_\d(u_\e - u_\d) \\
&- (|\n|_\e - |\n|_\d)(u_\e^2) - |\n|_\d( u_\e^2 - u_\d^2).
\end{split}\]
Testing with $(u_\e - u_\d)$ we further obtain
\[
\begin{split}
\frac12 \frac{d}{dt} |u_\e - u_\d|_2^2 &
= \! \int (u_\e - u_\d)^2|\n|_\e u_\e +  \int u_\d(u_\e - u_\d) (|\n|_\e - |\n|_\d) u_\e 
+  \int u_\d(u_\e - u_\d)|\n|_\d(u_\e - u_\d) \\
&-  \int (u_\e - u_\d)(|\n|_\e - |\n|_\d)(u_\e^2) -  \int (u_\e + u_\d)(u_\e - u_\d)|\n|_\d( u_\e - u_\d),
\end{split}
\]
where in the last term we swapped $|\n|_\d$ onto $( u_\e - u_\d)$. We
see that the third term $u_\d(|\n|_\e - |\n|_\d) u_\e (u_\e - u_\d)$ cancels with part of the last,
and we have, using the same trick as above:
\[
- \int u_\e (u_\e - u_\d)|\n|_\d( u_\e - u_\d) \leq -  \frac12 \int u_\e |\n|_\d( u_\e - u_\d)^2 
= -\frac12  \int  ( u_\e - u_\d)^2 |\n|_\d(u_\e)  \lesssim |u_\e - u_\d|_2^2,
\]
in view of the uniform bound on $u_\e$ in $H^m$. The rest of the terms
can be estimated using \eqref{e:diff}. Note that $m > d/2+1$ and being
an integer, $m\geq 2$ for any dimension $d\geq 1$. So, $H^m
\hookrightarrow H^2$. We thus have
\[
\begin{split}
\frac{d}{dt} |u_\e - u_\d|_2^2 & \leq C(  |u_\e - u_\d|_2^2 +  |\d- \e | |u_\e - u_\d|_2),
\end{split}
\]
where $C$ depends only on the initial conditions and other absolute
dimensional quantities, but on $\e$ and $\d$. Given that the solutions
start with the same initial condition, the Gr\"onwall inequality implies
\[
|u_\e (t) - u_\d(t) |_2 \leq C | \d- \e | (e^{Ct} - 1)
\]
for all $t<T$. This proves the claim.

So, the family $u_\d$ converges strongly to some $u$ in all $C([0,T);
H^{m'})$, $m'<m$. Moreover, $\p_t u_\d$ converges distributionally to
$\p_t u$, and in view of the uniform bound in $H^{m-1}$, it does so strongly
in any $H^{m'-1}$. This shows that the limit $u$ solves \eqref{e:main}
classically with initial condition $u_0$. Passing also to a weak limit
for a subsequence shows that $u\in C_w([0,T); H^m)$ which is the space of weakly
continuous $H^m$-valued functions. The argument to prove strong
continuity in $H^m$ now follows line by line that of \cite[Theorem
3.5]{mb-book} as we have all the same estimates. Finally, $u \in
C^1([0,T); H^{m-1})$ follows as before for $u_\d$, directly from the
equation.

Note that for the solution $u$ that we constructed, the maximum principle proved earlier for
$u_\d$ still holds. The  argument is the same, due to the positivity of the kernel.
\end{proof}

\subsection{A Beale-Kato-Majda criterion}

We now state the classical BKM criterion for our
model.  
\begin{theorem}[A Beale-Kato-Majda criterion]\label{thm:BKM}
  Suppose 
\[u\in C([0,T);H^m(\O^d)) \cap C^1([0,T); H^{m-1}(\O^d))\] 
is  a positive solution to \eqref{e:main}, where $m > d/2
  +1$. Suppose also that
\begin{equation}\label{weakBKM}
\int_0^T  | \nabla u (t) |_\infty \,dt < \infty.
\end{equation}
Then $u$ can be extended beyond time $T$ in the same regularity class.
\end{theorem}
\begin{remark} \label{rem:bkm-w2}
We will see that 
\(\int_0^T  | |\n| u (t) |_\infty dt < \infty\)
is also a  BKM criterion. 
\end{remark}
The proof relies on a version of the log-Sobolev inequality \`a la
Br\'ezis \cite{brezis_benilan} adapted to our setting.  In order to state it, let us first
recall some definitions. Let $u\in L^2$. Then $u$ admits a classical
Littlewood-Paley decomposition $u = \sum_{q = -\infty}^\infty
u_q$. Let us denote the large-scale part by $u_{<0} = \sum_{q<0} u_q$
and the small-scale part by $u_{\geq 0} = \sum_{q \geq 0} u_q$. Recall the
classical (homogeneous) Besov norm $\dot{B}^{s}_{r, \infty}$:
\[
\|u\|_{\dot{B}^{s}_{r, \infty}} = \sup_{q\in\Z} 2^{sq} | u_q|_r.
\]
\begin{lemma}[A log-Sobolev inequality]
    \begin{equation} \label{e:log-sob} |\n u|_\infty + ||\n| u|_\infty
    \lesssim |u|_2 + \| u_{\geq 0} \|_{{\dot B}^{1}_{\infty,\infty}} ( 1+ \log_+ \|u\|_m ) + 1.
\end{equation} 
\end{lemma}
\begin{proof}
  Bernstein's inequalities imply $| \n u_q |_\infty \sim 2^q
  |u_q|_\infty$, uniformly for all $q$. In particular, we get
\[
|\n u|_\infty + ||\n| u|_\infty \lesssim \sum_{q=-\infty}^\infty 2^q | u_q|_\infty.
\]
Then for $q <0$ we use another Bernstein inequality: $2^q |
u_q|_\infty \lesssim 2^{q(1+ \frac{d}{2})} |u_q|_2$. Clearly, $
|u_q|_2\leq |u|_2$ and thus $\sum_{q=-\infty}^{-1} 2^q | u_q|_\infty
\lesssim |u|_2$. For the small-scale component, we obtain, for every $Q
\geq 0$,
\[
\begin{split}
  \sum_{q=0}^\infty 2^q | u_q|_\infty & = \sum_{q=0}^Q 2^q | u_q|_\infty + \sum_{q=Q+1}^\infty 2^q | u_q|_\infty 
\leq Q\| u_{\geq 0} \|_{{\dot B}^{1}_{\infty,\infty}} +  \sum_{q=Q+1}^\infty 2^{q(1+ \frac{d}{2}-m)} 2^{qm} | u_q|_2 \\
  &\leq Q\| u_{\geq 0} \|_{{\dot B}^{1}_{\infty,\infty}} + 2^{-Q(m-\frac{d}{2}-1)} \| u\|_m.
\end{split}
\]
Minimizing the above over $Q$ (and recalling that $m>d/2+1$), the small-scale component is bounded by
$$\| u_{\geq 0} \|_{{\dot B}^{1}_{\infty,\infty}} \left( 1+ \log\frac{\|u\|_m}{\| u_{\geq 0} \|_{{\dot B}^{1}_{\infty,\infty}} } \right).$$
One observes next that $-x\log x\leq 1$ on $\R_+$ and that $\log y\leq \log_+ y$.
Thus the small-scale component is overall bounded by $1+\| u_{\geq 0} \|_{{\dot B}^{1}_{\infty,\infty}} ( 1+ \log_+ \|u\|_m )$, as claimed.
\end{proof}
We now turn to the proof of Theorem~\ref{thm:BKM}. 
\begin{proof}[Proof of Theorem~\ref{thm:BKM}]
Let $u$ be such that 
\begin{equation}\label{eq:BKM}
 \int_0^{+\infty} \| u_{\geq 0} (t) \|_{{\dot B}^{1}_{\infty,\infty}} dt < +\infty. 
\end{equation}
According to the classical Bernstein inequalities, $| \n u_q
|_\infty \sim 2^q |u_q|_\infty$, uniformly for all $q$, and by
continuity of the Littlewood-Paley projections, $| \n u_q |_\infty
\lesssim | \n u |_\infty$. Hence \eqref{weakBKM} implies
\eqref{eq:BKM}.  Similarly, since there exist $|\n|$-versions of
Bernstein's inequalities: $| |\n| u_q |_\infty \sim 2^q |u_q|_\infty$,
the condition in Remark~\ref{rem:bkm-w2} also implies \eqref{eq:BKM}.

Performing exactly the same estimates as on Step 3 of
  the proof of Theorem~\ref{thm:local} but now with $|\nabla|$ instead of $|\n_\d|$, we arrive at
  the following a priori bound
\begin{equation}\label{aprND}
\p_t \| u\|_m^2 \lesssim \|u\|_m^2 \sum_{i=1}^M |\n u|_\infty^{1-\mu_i} | |\n| u |_\infty^{\mu_i} ,
\end{equation}
where $M$ depends on $m$ and each of the $\mu_i$ satisfies  $0\leq \mu_i \leq 1$. Indeed, in \eqref{eq:work}, the estimate of the symmetrized term of highest order gives $\mu_i=1$.
The estimate of the first sum gives $\mu_i=|\alpha|/(m-1)$ for $0\leq \alpha <s$ and the specific terms from the last sum are dealt with simply with $\mu_i=0$.

Combining \eqref{aprND} with the log-Sobolev inequality~\eqref{e:log-sob}, we arrive at
\begin{equation}\label{aprND2}
  \p_t \| u\|_m^2 \lesssim |u|_2 \|u\|_m^2 + \| u_{\geq 0} \|_{{\dot B}^{1}_{\infty,\infty}} ( 1+ \log_+ \|u\|_m )\|u\|_m^2 + \|u\|_m^2.
\end{equation}
Applying Gr\"onwall's lemma twice, one gets a double-exponential estimate of the form:
$$\log\left(\frac{\log\|u(t)\|_m}{\log \|u_0\|_m + t +|u_0|_2t + \int_0^t  \| u_{\geq 0} \|_{{\dot B}^{1}_{\infty,\infty}} }\right) \leq \int_0^t  \| u_{\geq 0} \|_{{\dot B}^{1}_{\infty,\infty}}.$$
Theorem~\ref{thm:BKM} follows immediately.
\end{proof}

\subsection{From local to global through regularity}\label{s:weak}

We now study the question of the global existence and regularity of
positive weak solutions starting from arbitrary $L^\infty$-data.

Suppose that we are given a classical solution $u \in C([0,T);H^m)
\cap C^1([0,T); H^{m-1})$ on the torus $\T^d$ which is strictly
positive $u>0$. Let $T^*$ be its maximal time of existence. We
will show that $T^* = \infty$. Let us assume, on the contrary, that it is
finite. In what follows we apply the De Giorgi regularization
result of \cite{ccv} to our model. Since $u$ is a classical solution, the formal passage from the $u$-equation \eqref{e:main}
to the $w=u^2$-equation \eqref{eq:w-lin} holds true. The active kernel $k$ given by \eqref{eq:kernel} is 
symmetric with respect to $(x,y)$ and satisfies
\begin{equation}\label{}
\frac{\L(t)^{-1}}{|x-y|^{d+1}} \leq |k(t,x,y)| \leq \frac{\L(t)}{|x-y|^{d+1}}, 
\end{equation}
for all $x \neq y$, $t>0$, and
\[
\L(t) = C_d \max\{ |u(t)|_\infty , |u^{-1}(t)|_\infty\}.
\]
By the max/min principle, we see that $\L(t)$ can be replaced with
$\L(0) = \L$, uniformly for all time, and thus depends only on
$|u_0|_\infty$.  Equation \eqref{eq:w-lin} is exactly of the kind
studied in~\cite{ccv}.  It was natural in \cite{ccv}, in the context of an Euler-Lagrange
problem, to assume the finiteness of the global energy, i.e. $w \in L^2(\R^d)$.
The main technical result of \cite{ccv} however uses no
such assumption and only requires $w$ to have locally finite energy
(Corollary 3.2 of \cite{ccv} gives a global $L^\infty$ bound in terms
of the global $L^2$ norm, which in our case is not necessary as $w$
remains bounded by the maximum principle). So, \cite{ccv} applies
verbatim to our periodic solutions, unfolded on the whole space
$\R^d$. Specifically, the result states that there exists an $\a >0$
which depends only on $\L$ and $d$ such that for any $0<t_0<T^*$ we have $
w \in C^{\a,\a}_{x,t}(\T^d \times (t_0,T^*))$ with the bound
\[
\|w\|_{C^{\a,\a}_{x,t}(\T^d \times (t_0,T^*))} \leq C(t_0, |w_0|_\infty, \L, d).
\]
Next, the Schauder estimates for integro-differential equations of type
\eqref{eq:w-lin} recently obtained in \cite{iss} readily imply
\[
\| w \|_{C^{1+\frac{\a^2}2 ,\frac{\a^2}{2}}_{x,t}(\T^d \times [t_0+\e,T^*))}  
\leq C ( \e,\| w\|_{C^{0,\a}_{x,t}(\T^d \times [t_0,T^*))},\L,d)
\leq C(\e,t_0, |w_0|_\infty, \L, d).
\]
By the Beale-Kato-Majda criterion, $w$, and hence $u$, can be
extended smoothly beyond $T^*$, resulting in a contradiction.

A further application of the bootstrap argument of  \cite{iss} readily implies higher regularity bounds \eqref{e:instant} for $w$ and hence for $u$. We thus have established the following result.

\begin{theorem}[Global well-posedness]\label{t:global} Under the assumptions of \thm{thm:local}, the solution exists globally in time. Furthermore, the solution is regularized instantly and satisfies the bounds \eqref{e:instant}.
\end{theorem}

\subsection{Weak solutions from positive bounded data}\label{s:weak2}

The bounds \eqref{e:instant} depend on the initial data only through the minimum and maximum value.
This allows us to construct weak solutions from arbitrary positive data in $L^\infty$ and for which
similar regularization properties will hold. However it is not obvious how to restore the initial data
and one should be cautious of the topology used for the limit $t\to0$.

\begin{theorem}[Global weak solution]\label{t:weak} For any initial data $u_0\in L^\infty(\T^d)$, $u_0>0$, there exists a global weak solution to \eqref{e:main} in the class
$$L^\infty(\R_+ \times \T^d) \cap L^2(\R_+; \dot H^{1/2}) \cap C^0(\R_+; L^2).$$ Its initial value~$u_0$ is realized in the sense of the $L^\infty$ weak$^*$ limit and in the strong $L^2$ sense.
The energy is conserved, the momentum $\int_{\T^d} u(x,t)dx$ is continuous on $\R_+$ and \eqref{eq:ML} is satisfied for any $(t,t')\in\R_+^2$.
Furthermore, $u$ satisfies the instant regularization estimates \eqref{e:instant} and for all $t>0$, the original equation \eqref{e:main} is satisfied in the classical sense.
\end{theorem}
\begin{remark}
If uniqueness was to fail in \thm{t:weak}, it could only do so at $t=0$. However, the continuity of the momentum at $t=0$ prevents any concentration of the $\dot H^{1/2}$ norm in our weak solution.
\end{remark}

\begin{proof}[Proof of \thm{t:weak}]
Let $u_0 \in L^\infty$ be such a positive initial condition. We start the construction by taking standard mollifications $(u_0)_\e$ of $u_0$. In view of \thm{t:global}, there exists a global classical
solution~$u_\e$ from each of those mollified initial conditions. Let us fix $T>0$. In accordance with \eqref{e:instant}, the
family $\{u_\e\}$ satisfies uniform regularity starting from any  $t_0>0$. Invoquing the Arzel\`{a}-Ascoli compactness theorem together with Cantor's diagonal argument to successively reduce the value of $t_0$, we can pass to the limit and find a classical solution on the interval $(0,T]$. In addition, as we already pointed out in the introduction, the energy equality on the whole time interval $[0,T]$ combines with the momentum law to ensure that~$u_\e$ belongs to $ L^\infty([0,T) \times \T^d) \cap L^2([0,T); H^{1/2})$ uniformly. As $L^\infty\cap H^{1/2}$ is an algebra, $w_\e=u_\e^2$ enjoys the same uniform property.
Passing to the weak limit in the Sobolev space, we conclude that 
\begin{equation}\label{e:uw-unif}
u,w \in  L^\infty([0,T) \times \T^d) \cap L^2([0,T); H^{1/2}).
\end{equation}
For all $t>0$, since we also have a limit in the classical sense, $w(t) = u(t)^2$. The only remaining problem is to restore the initial data and a weak formulation of the equation near $t=0$.
Indeed, once a solution is constructed on $[0,T]$, one can invoque  \thm{t:global} one last time, starting with the smooth initial data $u(T)$ and therefore claim the existence of a global solution on $\R_+$.

\bigskip
Let us restore the initial data for $w$ first. Let us write \eqref{eq:w-lin} in the weak form, for the smoothed solutions: for any $\phi \in C^\infty([0,T) \times \T^d)$,
\begin{multline}\label{e:ww}
\int_{\T^d} w_\e(x,t) \phi(x,t) dx - \int_{\T^d} w_\e(x,0) \phi(x,0) dx - \int_0^t \int_{\T^d}  w_\e(x,t) \p_t \phi(x,s) dx ds \\
= \frac{1}{2}\int_0^t \int_{\T^d} k(s,x,y) (w_\e(y,s) - w_\e(x,s))(\phi(x,s) - \phi(y,s)) dx dy ds.
\end{multline}
Let us write \eqref{e:ww} as $A_\e - B_\e - C_\e = D_\e$ with the respective designation of each term. As we pass to the limit $\e \ra 0$, we trivially have $A_\e \ra A_0$.
The convergence $B_\e \ra B_0$ holds because $(u_0)_\e \ra u_0$ strongly in $L^2$ and hence $w_\e(t=0) = (u_0)_\e^2 \ra w_0$ in $L^1$ and thus weakly.
The convergence $C_\e \ra C_0$ results from the uniform $L^\infty$ bound on $w_\e$ (in an arbitrary small neighborhood of $t=0$) joined with the uniform convergence on the rest of the time interval.
Finally, one splits the last term as:
\begin{equation}\label{e:D}
D_\e = \int_\d^t + \int_0^\d = D_{\e,\d}' + D_{\e,\d}''.
\end{equation}
We have $D_{\e,\d}' \ra D_{0,\d}'$ by classical convergence and, by Cauchy-Schwarz and~\eqref{e:uw-unif},
\begin{equation}\label{e:D2}
|D_{\e,\d}''| \leq \| w_\e\|_{L^2([0,\d);H^{1/2})} \|\phi\|_{L^2([0,\d);H^{1/2})} \leq C \sqrt{\d}
\end{equation}
uniformly in $\e$. This shows that $D_\e \ra D_0$. We have thus proved that \eqref{e:ww} is satisfied in the limit as $\e \ra 0$, i.e. for $w$ itself. Taking then $\phi$ independent of $t$ and passing to $t\ra 0$ shows that $w(t) \rightharpoonup w_0$ weakly$^*$ in $L^\infty$. 

\bigskip
At this point let us make a cautionary remark because $u(t)^2 \rightharpoonup u_0^2$ does not imply $u(t) \rightharpoonup u_0$ in general. A simple example is provided by the sequence $u_n = 1+ \frac12 r_n$, where $r_n(x)=\operatorname{sign}(\sin 2^n\pi x)$ are Rademacher functions. Then $u_n>0$, $u_n \rightharpoonup 1$ and yet $u_n^2 \rightharpoonup \frac54 \neq 1^2$. A progressively mollified sequence $(u_n)_{1/n^n}$ would provide a similar counter-example in the class $C^\infty$, as in our case. However, one can observe on this example that the function $\sqrt{5/4}$, whose square is the limit of squares, dominates the limit of $u_n$ itself, which is $1$. This is true in general.

\begin{lemma}
Suppose that a sequence of functions $\{u_n\} \ss L^\infty$, bounded away from zero, enjoys both limits $u_n \rightharpoonup u'$ and $u_n^2 \rightharpoonup u_0^2$ in the weak$^*$ topology. Then $u_0 \geq u'$.
\end{lemma}
\begin{proof}
The proof of this lemma is simple. Let $\phi >0$ be a test function. Then trivially
\[
\int (u_n - u_0)^2 \phi dx \geq 0,
\]
for all $n$. Let us expand,
\[
\int (u_n - u_0)^2 \phi dx = \int (u_n^2 -2 u_n u_0 + u_0^2) \phi \ra \int (u_0^2 - 2 u' u_0 + u_0^2)\phi = 2\int u_0(u_0 - u') \phi  \geq 0.
\]
Since $u_0 >0$ and the above holds for an arbitrary $\phi>0$ in $L^1$, the lemma follows.
\end{proof}

\bigskip
Let us go back to restoring the intial condition for $u$.
Reverting to the original  equation~\eqref{e:main3}, its weak formulation for the smooth sequence reads, with test functions independent of $t$:
\begin{multline}\label{e:wu}
\int_{\T^d} u_\e(x,t) \phi(x) dx - \int_{\T^d} u_\e(x,0) \phi(x) dx = \\
\frac12 \int_0^t \int_{\T^d}  \phi(x) (u_\e(y,s)-u_\e(x,s))^2 K_{\per}(y-x) dy dx ds \\
+ \frac12 \int_0^t \int_{\T^d} u_\e(x,s) (\phi(x) - \phi(y))(u_\e(y,s) - u_\e(x,s)) K_{\per}(y-x) dy dx ds.
\end{multline}
Passing to the limit on the left-hand side and in the last integral on the right presents no difficulty as one can use estimates similar to \eqref{e:D}-\eqref{e:D2}.

However there are no a-priori bounds that guarantee the smallness near the time $t=0$ of the first integral on the right-hand side. Specifically, a possible concentration of the $H^{1/2}$ norm near $t=0$ (or equivalently by \eqref{eq:ML}, an initial discontinuity in the first momentum) could prevent \eqref{e:wu} from withstanding the limit.

Using only the positivity  of the first integrand of the right-hand side,
we obtain instead, for all $\phi \geq 0$:
\begin{multline}\label{e:wu2}
\int_{\T^d} u(x,t) \phi(x) dx - \int_{\T^d} u_0(x) \phi(x) dx \geq \\
\frac12 \int_0^t \int_{\T^d} u(x,s) (\phi(x) - \phi(y))(u(y,s) - u(x,s)) K_{\per}(y-x) dy dx ds.
\end{multline}
In the limit $t\ra 0$, the right-hand side of the previous inequality vanishes by~\eqref{e:D2}, and hence, any weak$^*$ limit of a subsequence of $(u(t))_{t>0}$ would converge to a function $u'$ satisfying $u'\geq u_0$ a.e.
On the other hand, by the lemma above, $u_0 \geq u'$. This proves that the weak$^\ast$ limit $u(t) \rightharpoonup u_0$ holds as $t\ra 0$ for any subsequence and therefore also
as a function of the continuous time parameter.
In particular, testing this weak$^\ast$ limit with $\phi\equiv 1$ ensures that the momentum $\int_{\T^d} u(x,t)dx$ is continuous at $t=0$.

\bigskip
Let us now recall that the $L^2$ norm is also continuous at time zero and therefore overall preserved. Indeed, one can for example use $\phi\equiv 1$ as a test function in the weak$^\ast$ limit $w(t) \rightharpoonup w_0$ or directly notice that $|u_\e(t)|_2 = |(u_0)_\e|_2$ for $\e>0$ and that this identity passes to the limit $\e\ra0$, on the left-hand side
because of the uniform convergence at time $t>0$ and on the right-hand side because it is a standard property of mollified sequences. Next, as $u_0\in L^2\subset L^1$ is an admissible test function
for the weak$^\ast$ convergence $u(t) \rightharpoonup u_0$, one gets
$$|u(t)-u(0)|_2^2 = 2|u_0|^2-2\int_{\T^d} u(x,t)u_0(x)dx \to 0.$$
Therefore $u\in C(\R_+;L^2)$.

\bigskip
Finally, to restore the weak formulation of the original equation, we first write it on a time interval $[t_0,t]$ with $t_0>0$, where it is also satisfied in the classical sense:
\begin{multline}\label{e:wu3}
\int_{\T^d} u(x,t) \phi(x,t) dx - \int_{\T^d} u(x,t_0) \phi(x,t_0) dx - \int_{t_0}^t \int_{\T^d} u(x,s) \p_t \phi(x,s) dx ds= \\
\frac12 \int_{t_0}^t \int_{\T^d} \phi(x,s) (u(y,s)-u(x,s))^2 K_{\per}(y-x)  dy dx ds \qquad\qquad\\
+ \frac12 \int_{t_0}^t \int_{\T^d}  u(x,s) (\phi(x,s) - \phi(y,s))(u(y,s) - u(x,s)) K_{\per}(y-x) dy dx ds.
\end{multline}
Passing to the limit as $t_0\ra0$ is now within reach. On the left-hand side, the $L^2$ continuity ensures the convergence of the middle term while the $L^\infty$ bound tackles the third term.
On the right-hand side, the last integral is controled in the same fashion as~\eqref{e:D2} because of~\eqref{e:uw-unif}. The troublesome first term on the right-hand side is
bounded effortlessly; for any $0<t_1<t_0$:
\begin{equation}\label{e:wu4}
\left| \int_{t_1}^{t_0} \int_{\T^d} \phi(x,s) (u(y,s)-u(x,s))^2 K_{\per}(y-x)  dy dx ds \right| \leq C \int_{t_1}^{t_0} \|u(s)\|_{\dot{H}^{1/2}}^2 ds.
\end{equation}
Applying~\eqref{eq:ML} on $[t_1,t_0]$, where the equation is already satisfied in a classical sense, the right-hand side equals $C\int_{\T^d} u(x,t_0) - u(x,t_1) dx$.
As the momentum is continuous, the right-hand side is arbitrarily small as $t_0\to0$, uniformly for $t_1\in[0,t_0]$.
One can thus first let $t_1\to 0$ in \eqref{e:wu4} and then pass to the limit $t_0\to0$ in~\eqref{e:wu3}.
Finally, taking $\phi\equiv1$ in the weak formulation of the equation shows that \eqref{eq:ML} still holds at $t=0$.
This finishes the complete construction of a weak solution on $[0,\infty)$ and the proof of \thm{t:weak}. 
\end{proof}

\subsection{Blowup in finite time for some smooth negative data}\label{s:blowup}

In view of the time reversibily property (TR) mentioned page \pageref{TR},  if
$u$ is a positive solution to \eqref{e:main}, then $-u(t^*-t)$ is a
negative solution. Thus starting with positive data $u_0 \in
L^\infty(\T^d) \bs C(\T^d)$ we obtain a solution $u$ which becomes
$C^\infty$ instantaneously. Then $-u(t^*)$ serves as negative initial data that develops a singularity at time $t= t^*$. 

\begin{corollary}[Finite time blow-up] For any $t^*>0$, there exists a negative initial condition  $u_0 \in C^\infty(\T^d)$, $u_0<0$ and there exists a classical solution to \eqref{e:main} that develops into a discontinuous  solution at time $t^*$ i.e. $u(t^*) \in L^\infty(\T^d) \bs C(\T^d)$.
\end{corollary}

\section{Long-time behavior}
\label{sec:ltb}

In this section we will show that the long-time dynamics of the model converges to a constant state consistent with the conservation of energy, namely,
\[
u(t,x) \underset{t\to+\infty}{\longrightarrow} \frac{\|u_0\|_{L^2}}{\sqrt{|\T^d|}}\cdotp
\]
As the solution is squeezed by the maximum and minimum principles, it is expected that such behavior would be a consequence of a steady decay of the amplitude. We will show indeed in Lemma~\ref{l:A} that the amplitude tends to zero exponentially fast. However, this first property alone still leaves room for non-trivial oscillations like a non-vanishing $|\n u|_\infty$. We will exclude such residual oscillations using the method that Constantin and Vicol \cite{cv} developed recently to prove, among other things, the global regularity  for the critical SQG equation.
In short, the long-time dynamics of our model is a convergence to a constant state, both at large and small scales.

Subsequently, we use the following notations:
$$m(t)=\min_{x\in\T^d} u(x,t), \quad M(t)=\max_{x\in\T^d} u(x,t), \quad A(t)=M(t)-m(t)$$
and $m_0=m(0)$, $M_0 = M(0)$, $A_0 = A(0)$.
\begin{lemma}[Exponential decay of space oscillations]\label{l:A} 
  Let $u \in L^\infty(\T^d \times [0,\infty))$ be a weak solution to
  \eqref{e:main2} with a positive initial condition as stipulated in \thm{t:weak}. Then $A(t) \leq A_0 e^{-c m_0 t}$ holds for all
  $t>0$ with some $c>0$ independent of $u$.
\end{lemma}
\begin{proof}
  This proof relies on an idea from \cite{fim}.  Let $\bar{x}(t)$
  denote a point where $M(t)= u(\bar{x}(t),t)$ and let
  $\underline{x}(t)$ denote a point where $m(t) = u(\underline{x}(t),t)$. 
Let us unfold the $2\pi$-periodic solution on $\R^d$ ; without loss of generality, we can assume that $\bar{x},\underline{x} \in [-\pi,\pi]^d$. 
The gradient $\nabla u(\cdot,t)$ vanishes both at $\bar{x}(t)$ and $\underline{x}(t)$.
Then we have, for instance in the viscosity sense of Crandall-Lions,
\begin{align*}
M'(t) & \le \int u(y)(u (y) -u(\bar{x})) K(y-\bar{x})dy \\
& \le m_0 \int_{|y-\bar{x}|\ge 1, |y-\underline{x}|\ge 1} (u (y) -u(\bar{x})) K(y-\bar{x})dy \\
& \le m_0 \int_{|y-\bar{x}|\ge 1, |y-\underline{x}|\ge 1} (u (y) -u(\bar{x})) 
\min \left(K(y-\bar{x}),K(y-\underline{x})\right) dy.
\end{align*}
We dropped the reference to time in the right-hand sides for readibility.
Similarly, we have 
\[ m'(t) \ge m_0 \int_{|y-\bar{x}|\ge 1, |y-\underline{x}|\ge 1} (u(y)- u (\underline{x})) 
\min \left(K(y-\bar{x}),K(y-\underline{x})\right) dy. \] 
This implies that 
$$A'(t) \le - m_0 A(t) \int_{|y-\bar{x}|\ge 1, |y-\underline{x}|\ge 1} 
\min \left(K(y-\bar{x}),K(y-\underline{x})\right) dy  \le -c m_0 A(t) $$
where 
\( c = \int_{|y|\ge 1 +\pi\sqrt{d}} \frac{c_d dy}{(|y|+\pi\sqrt{d})^{d+1}} .\)
Lemma~\ref{l:A} follows by Gr\"onwall's lemma.
\end{proof}
\begin{theorem}[Asymptotic behavior at all scales]\label{thm:asympt}
  Let $u \in L^\infty(\T^d \times [0,\infty))$ be a weak solution to
  \eqref{e:main} with a positive initial condition as stipulated in \thm{t:weak}.  Then $m(t)$ is increasing, $M(t)$ is decreasing
  and both $A(t)$ and $|\n u(t)|_\infty$ decay to zero exponentially
  fast. Consequently, the solution converges to a constant at an
  exponential rate.
\end{theorem}
\begin{remark}Combining this asymptotic with~\eqref{eq:ML} allows us to compute the total increase of momentum as the defect in the Cauchy-Schwartz inequality for $u_0$ :
$$\int_0^\infty \|u(t)\|_{\dot H^{1/2}}^2 dt = \sqrt{|\T^d|}|u_0|_2-\int_{\T^d}u_0(x)dx.$$
Lemma~\ref{l:A} indicates that the stabilization of the largest scales starts right away. The following proof will establish that the time-scale of the transitory regime that precedes the
stabilization of the lowest scales does not exceed $T^*\simeq\frac{1}{m_0}\log_+(A_0/Cm_0)$.
\end{remark}
\begin{proof}
Let us take the gradient of the integral form \eqref{e:main2} and
multiply by $\n u$. After elementary manipulations we obtain pointwise (the integrals being understood as principal values):
\[
\begin{split}
\frac12 \p_t |\n u(x,t)|^2 &+ \n u(x)\cdot \int (\n u(x) - \n u(x+z)) u(x+z) K(z)dz \\
& = \n u(x) \cdot \int  (u(x+z) - u(x)) \n u(x+z) K(z)dz.
\end{split}
\]
Inside the first integral let us use the identity
\begin{equation}\label{}
\begin{split}
\n u(x)\cdot (\n u(x) - \n u(x+z)) u(x+z) &= \frac12 [ |\n u(x)|^2 - |\n u(x+z)|^2 ]u(x+z)\\
& + \frac12 |\n u(x) - \n u(x+z)|^2 u(x+z).
\end{split}
\end{equation}
Assuming that $x$ is the point of the maximum of $|\n u|$, the first term
on the right-hand side of this last identity is non-negative, while in the
second term we can simply use the bound from below $u(x+z)\geq m$.
We thus obtain the estimate
\[
\begin{split}
\p_t |\n u(x,t)|^2 &+ m_0\int  |\n u(x+z) - \n u(x)|^2 K(z)dz \\
& \leq  2\int \n u(x) \cdot  \n u(x+z) \left(u(x+z) - u(x)\right) K(z) dz.
\end{split}
\]
Let us denote the integrals respectively by $I$ and $J$:
\begin{equation}\label{e:bal}
\begin{split}
\p_t |\n u(x,t)|^2 &+ m_0 I \leq J.
\end{split}
\end{equation}

\bigskip
We split the integral $J$ depending on whether $|z|<\r$ or $|z|>\r$
and write $J = J_{<\r}+ J_{>\r}$. We estimate $J_{>\r}$ after rewriting it in the following form
\[
\begin{split}
J_{>\r} & =   2 \n u(x) \cdot\int_{|z| > \r}  \n_z (u(x+z)-u(x)) (u(x+z) - u(x)) K(z)dz \\
& =  \n u(x) \cdot \int_{|z| > \r}   \n_z (u(x+z)-u(x))^2  K(z)dz\\
& = (d+1) \n u(x) \cdot \int_{|z| > \r} z  (u(x+z)-u(x))^2  \frac{c_d dz}{|z|^{d+3}} 
+  \n u(x) \cdot\int_{|z| = \r} \nu_z (u(x+z)-u(x))^2  \frac{c_d d\s (z) }{\rho^{d+1}}\cdotp
\end{split}
\]
Thus,
\[
|J_{>\r}| \leq c_2  |\n u|_\infty \frac{A^2}{\r^2}.
\]
As to $J_{<\r}$, we use the first order Taylor formula for the increment of $u$:
\[
\begin{split}
J_{<\r} & = 2  \int_0^1 \operatorname{p.v.}\left(\int_{|z|<\r} \n u(x) \cdot  \n u(x+z) \n u(x+\l z) \cdot z K(z) dz\right) d\l \\
& =  2 \int_0^1 \int_{|z|<\r} \n u(x) \cdot  \n u(x+z) (\n u(x+\l z) - \n u(x)) \cdot z K(z) dz d\l \\
&\qquad+ 2\int_0^1 \int_{|z|<\r} \n u(x) \cdot  (\n u(x+z) - \n u(x)) \n u(x) \cdot z K(z) dz d\l\\
& = J_{<\r}^1 + J_{<\r}^2.
\end{split}
\]
We get an upper bound of $J_{<\r}^1$ by Cauchy-Schwarz:
\[
\begin{split}
| J_{<\r}^1 | & \leq 2|\n u|_\infty^2  \int_0^1 \left(\int_{|z|<\l \r} \frac{|\n u(x+ z) - \n u(x)|}{|z|^d} c_d dz\right) d\l \\
& \leq 2|\n u|_\infty^2 \int_{|z|< \r} |\n u(x+ z) - \n u(x)|\sqrt{K(z)}  \cdot \frac{\sqrt{c_d}}{|z|^{\frac{d-1}{2}}} dz \\
&\leq  2|\n u|_\infty^2 \sqrt{c_d\omega_d I \r} \leq \frac{m_0}{4} I + c_3 \frac{|\n u|_\infty^4}{m_0} \r.
\end{split}
\]
The estimate for $J_{<\r}^2$ is completely analogous (without the $\l$-integral). Thus,
\[
J_{<\r} \leq \frac{m_0}{2} I + c_4  \frac{|\n u|_\infty^4}{m_0} \r.
\]
Incorporating the estimates already obtained back into \eqref{e:bal}, we arrive at
\[  \p_t |\n u(x,t)|^2 + \frac{m_0}{2} I \leq  c_2  |\n u|_\infty \frac{A^2}{\r^2}+c_4  \frac{|\n u|_\infty^4}{m_0} \r .\]
The choice $\rho= (m_0A^2)^{1/3}/|\n u|_\infty$ optimizes the right-hand side:
\begin{equation}\label{e:bal2}
  \p_t |\n u(x,t)|^2 + \frac{m_0}{2} I \leq  c_5  \left(\frac{A}{m_0}\right)^{\frac23} |\n u|_\infty^3 .
\end{equation}

\bigskip
Next, we estimate $I$ from below using an argument analogous to that
of P.~Constantin and V.~Vicol \cite{cv} but in which we incorporate the amplitude. For an arbitrary $r>0$, we write
\begin{gather*}
I = \int |\n u(x) - \n u(x+z)|^2 K(z)dz \geq \int_{|z|>r} |\n u(x) - \n u(x+z)|^2 K(z)dz  \\
I \geq  \int_{|z|>r} |\n u(x)|^2 K(z)dz  - 2 \n u(x) \cdot \int_{|z|>r} \n u(x+z) K(z)dz
\end{gather*}
therefore
\[
\begin{split}
 I& \geq c_6 \frac{|\n u(x)|^2}{r} - 2 \n u(x) \cdot \int_{|z|>r} \frac{ \n_z( u(x+z) - u(x)) }{|z|^{d+1}} c_ddz \\
  &= c_6 \frac{|\n u(x)|^2}{r} -2(d+1) \n u(x) \cdot \int_{|z|>r} z (u(x+z) - u(x)) \frac{c_d dz}{|z|^{d+3}} \\
& \qquad +2 \n u(x) \cdot \int_{|z|=r} \nu_z (u(x+z) - u(x))\frac{c_d d\s (z)}{r^{d+1}}\cdotp
\end{split}
\]
It follows $$ I \geq c_6 \frac{|\n u(x)|^2}{r} - c_7|\n u(x)| \frac{A}{r^2}$$
and choosing $r = \frac{2c_7 A}{c_6 |\n u(x)|}$ provides
\begin{equation}\label{I}
I \geq 2 c_8 \frac{| \n u(x,t)|^3}{A}\cdotp
\end{equation}

\bigskip
As $x$ is a point of maximum of $\nabla u(t)$, combining \eqref{e:bal2} with \eqref{I} yields
\begin{equation}\label{e:bal3}
\p_t |\n u|_\infty^2 +  c_8 \frac{ m_0 }{A} | \n u |_\infty^3 \leq c_5  \left(\frac{A}{m_0}\right)^{\frac23} |\n u|_\infty^3 .
\end{equation}
  In view of \lem{l:A}, there exists a time $T^*\simeq\frac{1}{m_0}\log_+(A_0/Cm_0)$ such that for $t \ge T^*$, 
\[ c_5  \left(\frac{A}{m_0}\right)^{\frac23} \le  c_8 \frac{ m_0 }{2A}\cdotp\]
This implies that for $t \ge T^*$,
\[\p_t |\n u|_\infty^2 +  c_8 \frac{ m_0 }{2A} | \n u |_\infty^3 \leq 0 \] 
or equivalently,
\[\p_t |\n u|_\infty +  c_8 \frac{ m_0 }{4A} | \n u |_\infty^2 \leq 0. \] 
Using the precise estimate from \lem{l:A}, we further get 
\[ \p_t |\n u|_\infty +  \frac{c_8 m_0 }{4A_0} e^{cm_0t} | \n u |_\infty^2 \leq 0.\]
This finally implies for $t\geq T^*$
\[ |\nabla u(t)|_\infty \le \frac{|\nabla u (T^*)|_\infty}{1 + \frac{c_9|\nabla u (T^*)|_\infty}{A_0} (e^{cm_0 t} - e^{cm_0 T^*})}
= \frac{|\nabla u (T^*)|_\infty e^{-cm_0 t}} {1 + \frac{c_9|\nabla u (T^*)|_\infty}{A_0} (1-e^{-cm_0 (t-T^*)}) } \cdotp
\]
The proof of Theorem~\ref{thm:asympt} is now complete. 
\end{proof}
 
\section{Local well-posedness in analytic classes} 
\label{par:analytic}

In the last section we plan to explore numerically what happens, in terms of existence or blow-up, to solutions when the initial condition is unsigned or even negative.
Since for such data we do not even have a local existence result in Sobolev spaces, we find it necessary to prove a generic local existence result, at least in analytic classes.
Since our numerical data obviously has a finite Fourier spectrum, it will ensure the minimal solid ground required to run the numerical scheme.

\medskip
We rely on a not so well-known fixed point theorem by Nishida and Nirenberg, \cite{nishida}. Let us recall it.
Let us define the following spaces, for $\r>0$,
\begin{equation}\label{}
X_\rho = \{ u: |u|_\rho =  \sum_{k \in \Z^3} e^{|k|\rho} | \hat{u}(k) | < \infty \},
\end{equation}
\begin{equation}\label{}
Y_\rho = \{ u:  \| u \|_{\rho} = | u |_\rho + | \n u |_\rho \}.
\end{equation}
Notice that $X_\rho$ is a Banach algebra. Also notice that for any $\rho'>\rho$ one has
\begin{equation}\label{D-bound}
| \n u |_\rho = | |\n| u |_{\r} \leq \frac{C}{\rho'-\rho} | u |_{\rho'},
\end{equation}
for some absolute $C>0$. 
\begin{theorem}[Nishida and Nirenberg, \cite{nishida}]
Suppose a functional $F=F(u)$ 
satisfies the following condition: $F: Y_{\r'} \ra Y_\r$ for all $\r'>\r$, and for any $R>0$ there is a $C_R$ such that for all $\|u,u_1,u_2\|_{\r'} <R$ and for all $\r<\r'$ the following holds:
\begin{itemize}
\item[(i)] $\| F(u) \|_{\r} \leq \frac{C_R}{\r'-\r} \|u\|_{\r'}$,
\item[(ii)] $\| F(u_1) - F(u_2) \|_{\r} \leq \frac{C_R}{\r'-\r} \|u_1-u_2\|_{\r'}$.
\end{itemize}
Then the Cauchy problem
\begin{equation}\label{e:nishida}
\begin{cases}
\frac{d u }{dt}= F(u),\\
u(0)  = u_0
\end{cases}
\end{equation}
is locally well-posed in the following sense: for any $u_0 \in Y_{\r'}$ and any $\r<\r'$, there exists $T$ and a unique solution $u \in C^1((-T,T); Y_\r)$ to \eqref{e:nishida}.
\end{theorem}

In our case, $F(u) =  u |\n| u - |\n|(u^2)$. Let us fix an $R>0$ and assume that $\|u\|_{\r'} <R$ and $\rho < \rho'$. Using the algebra property of $Y_\r$ and \eqref{D-bound} we estimate (absolute constants may change from line to line) 
\begin{equation*}\label{}
| F(u) |_{\rho} \leq | u |_{\rho} | |\n| u |_{\rho}  + \frac{C}{\rho'-\rho} |u^2|_{\r'} \leq \frac{CR}{\rho'-\rho} |u|_{\r'}\leq \frac{CR}{\rho'-\rho} \|u\|_{\r'},
\end{equation*}
and
\begin{align*}\label{}
| \n F(u) |_{\rho} &\leq |\n (u |\n| u)|_\r + 2 ||\n|(u \n u)|_\r \\&\leq \frac{C}{\rho'-\rho} (|u|_{\r'} ||\n| u|_{\r'} + |u|_{\r'} |\n u|_{\r'})  \leq \frac{C}{\rho'-\rho} \|u\|_{\r'}^2 \leq \frac{CR}{\rho'-\rho} \|u\|_{\r'}.
\end{align*}
Thus, $\| F(u) \|_{\r} \leq \frac{CR}{\r'-\r} \|u\|_{\r'}$.  The Lischitzness condition is verified analogously.

\section{Numerical simulations and perspectives}
\label{par:numerics}

In this last section, let us illustrate briefly with some numerical simulations the main analytical results of our paper, in dimension $d=1$. 

\subsection{Smooth positive data}\label{s:SPD}

Positive smooth data satisfy the assumptions of our main Theorem. We computed the solution using a direct integration of the singular-integral form of \eqref{e:main2}.
Given a time step $\tau>0$ and a spatial discretization sequence $(x_i)_{1\leq i\leq N}$ of $\T^1\simeq[-\pi,\pi]$ with a uniform mesh size $\delta=x_{j+1}-x_j$,
the value $U_{i,k}$ representing $u(k\tau,x_i)$ is computed with a forward numerical scheme:
\begin{equation}\label{numscheme}
U_{k+1,i}=U_{k,i}+\tau\delta \sum_{j\neq i} \frac{(U_{k,j}-U_{k,i})U_{k,j}}{4 \sin^2(\frac{x_j-x_i}{2})}\cdotp
\end{equation}
The code runs well as long as one makes sure to respect a CFL ($\tau$ is small enough) and that each term stays well within the limit of exact computer arithmetic (to compute
accurately the compensations i.e.~the principal value).

Figure \ref{SimA}  shows the evolution from $u_0(x)=2+\sin(x)+\frac{3}{10}\cos(5x)$ and the profile of the solution for various time-slices. As expected in view of Theorems \ref{t:global} and \ref{thm:asympt},
the solution converges rapidly towards the theoretical limit
$$\frac{\|u_0\|_{L^2}}{\sqrt{2\pi}}=\frac{3}{10}\sqrt{\frac{101}{2}}\simeq 2.1319.$$
The maximum and minimum principle are clearly visible.

\bigskip\par\noindent
\begin{minipage}{.98\linewidth}
\begin{center}
\captionsetup{type=figure}
\includegraphics[width=.48\textwidth,height=.3\textwidth]{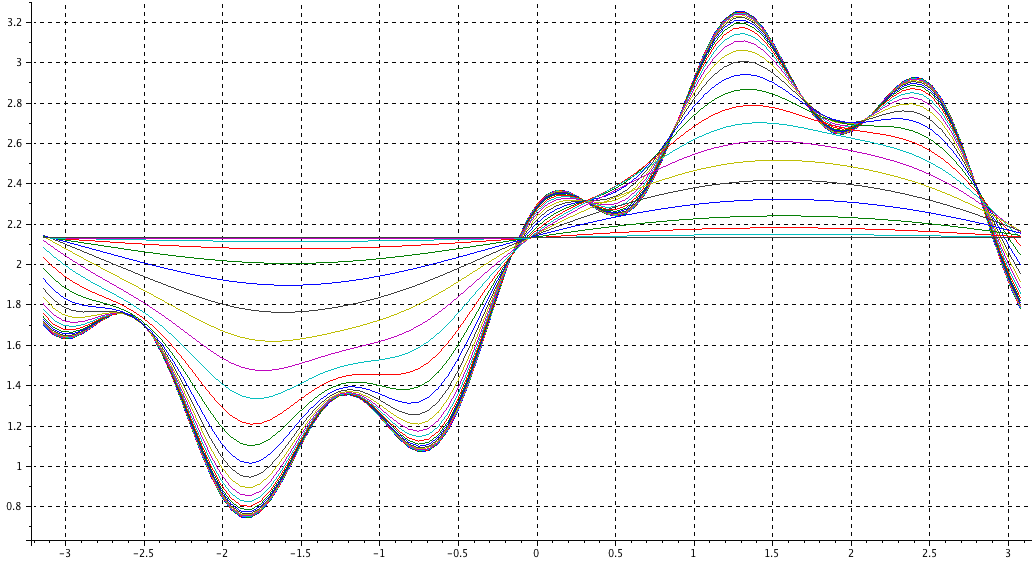}\hspace{\stretch{1}}\includegraphics[width=.48\textwidth,height=.3\textwidth]{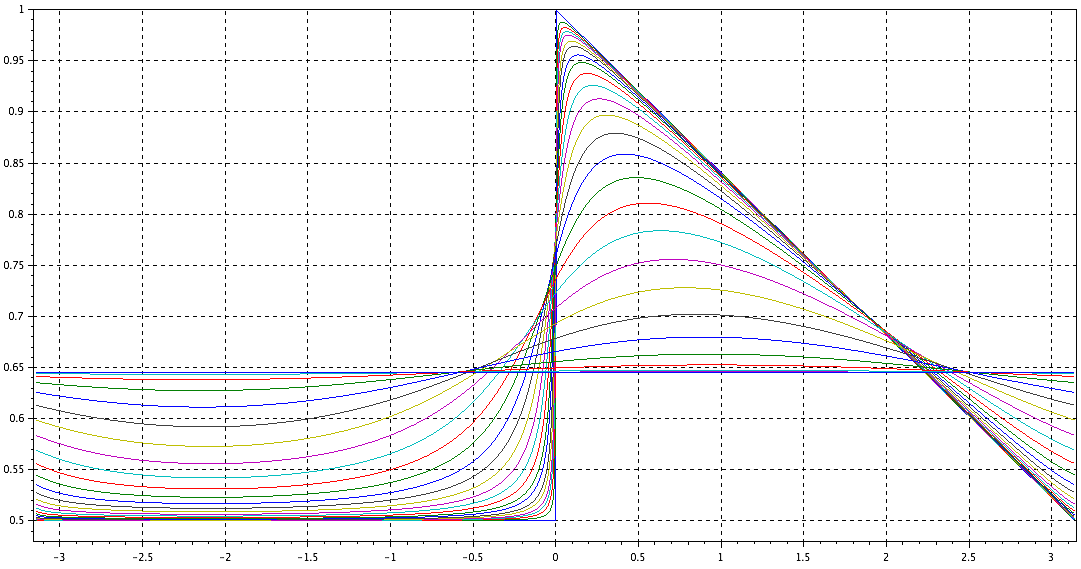}
\captionof{figure}{Evolution from various data. Left: $u_0(x)=2+\sin(x)+\frac{3}{10}\cos(5x)$.
Right: $u_0(x)=1/2$ if $x\in[-\pi,0)$ and $1-\frac{x}{2\pi}$ if $x\in(0,\pi]$.}\label{SimA}
\end{center}
\end{minipage}

\subsection{Evolution from discontinuous data}

Figure \ref{SimA}  shows also the evolution from a positive initial condition with a single discontinuity:
$$u_0(x)=\begin{cases}1/2 & \text{if }x\in[-\pi,0)\\ 1-\frac{x}{2\pi} & \text{if } x\in(0,\pi].\end{cases}$$
Even though the numerical scheme \eqref{numscheme} is not particularly refined, we did not observe any numerical instability. This observation is consistent with the instant smoothing effect of our \thm{t:weak}
and is a timid but positive point to support the conjecture that only one weak solution arises from positive bounded data.
We pushed the test further and Figure \ref{SimA3} shows the evolution from non-smooth numerical data cooked up to simultaneously be non-derivable at $\pm\pi$ and $-1$, have a jump at 0 followed by a chirp-like oscillation $\sqrt{x}\sin(1/x)$. The energy spectra (modulus of Fourier modes, in log-log scale) of various time-slices attest to the reality of the smoothing effect over 2 decades and beyond.

\bigskip\par\noindent
\begin{minipage}{\linewidth}
\begin{center}
\captionsetup{type=figure}
\includegraphics[width=.8\textwidth,height=.45\textwidth]{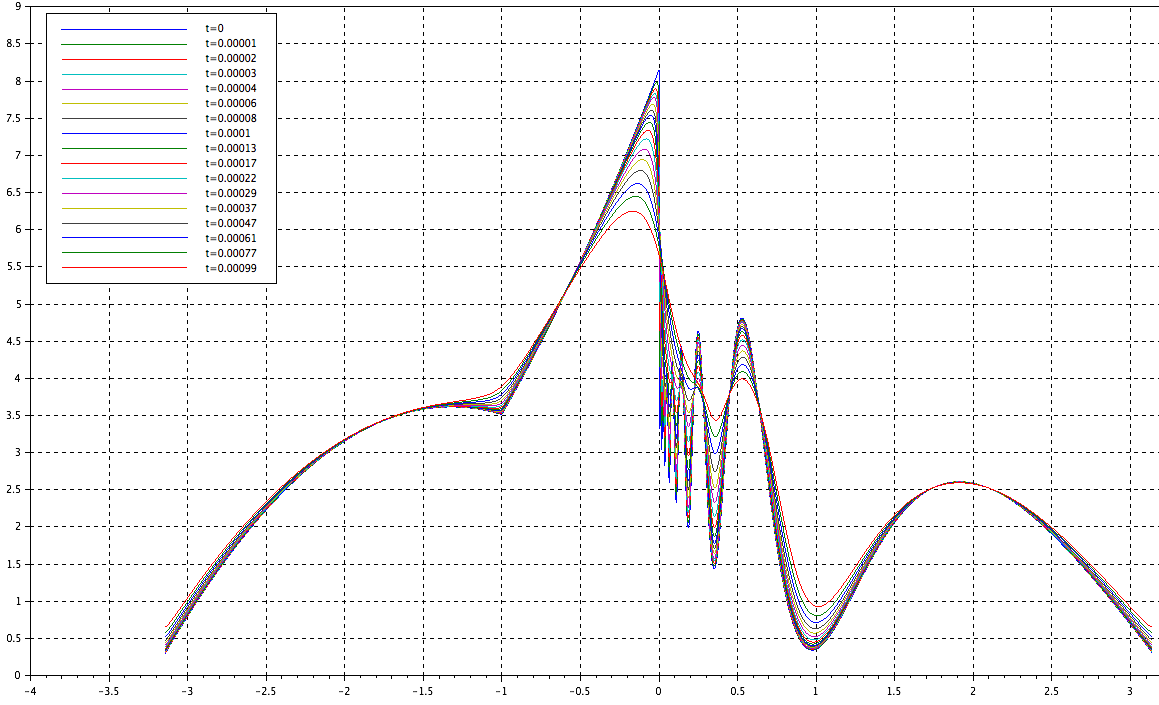}\\[3pt]\includegraphics[width=.83\textwidth,height=.45\textwidth]{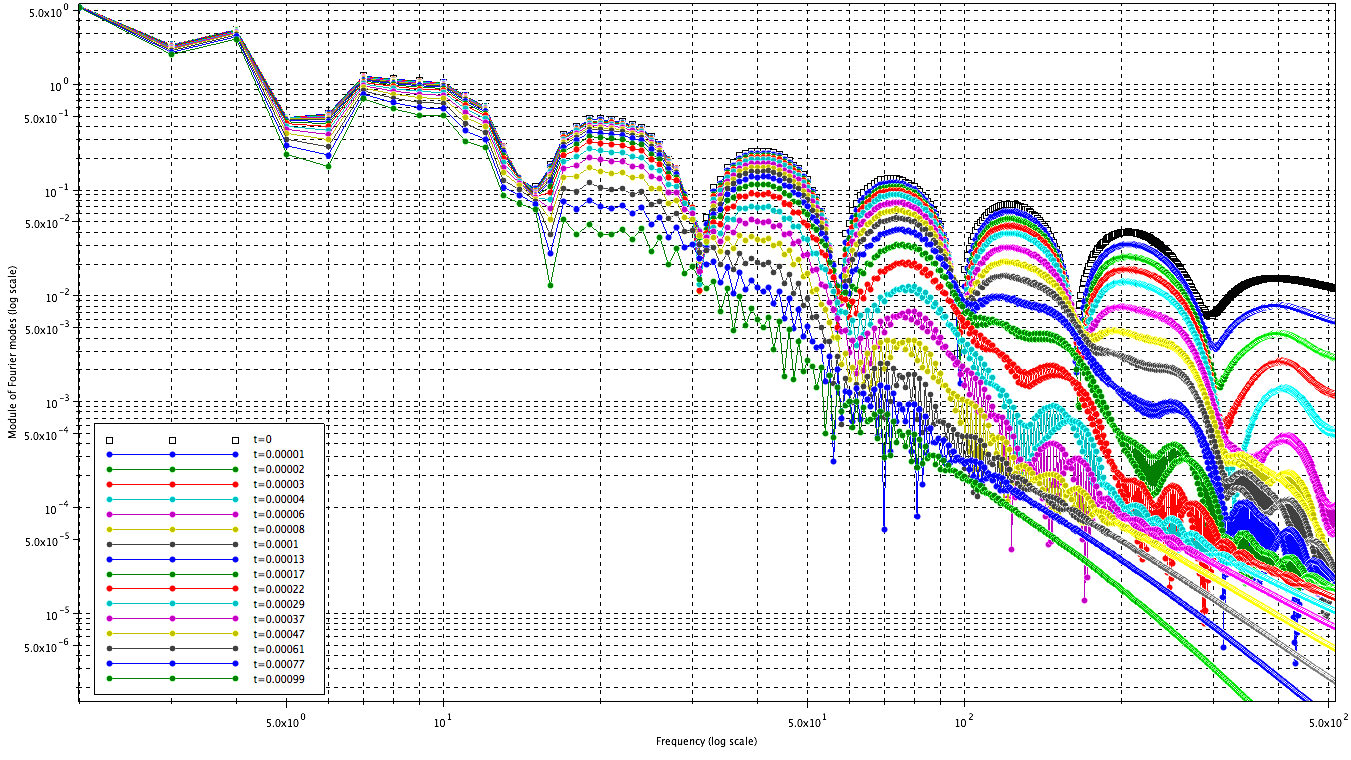}
\captionof{figure}{Evolution from non-smooth data (non-derivable at $\pm\pi$ and $-1$, a jump at 0 followed by a chirp $\sqrt{x}\sin(1/x)$) and energy spectrum.}
\label{SimA3}
\end{center}
\end{minipage}

\subsection{Evolution from negative data}
Figure \ref{SimD} shows the evolution starting from $$u_0(x)=-2-\sin(x)-\frac{3}{10}\cos(5x),$$ i.e. the opposite of the positive data of paragraph \ref{s:SPD}.
The first time-slices remain smooth as one could expect from the local well-posedness in analytic classes of Section \ref{par:analytic}. The (AMP)  principle is clearly visible as well. The last time-slice shows the emergence of the first singularity: a highly oscillatory singularity forms around the minimum value of $u_0$, while the rest of the solution remains relatively stable. A secondary instability is also barely discernible near the second lowest value of the data.

\medskip\par\noindent
\begin{minipage}{\linewidth}
\begin{center}
\captionsetup{type=figure}
\includegraphics[width=.9\textwidth,height=.45\textheight]{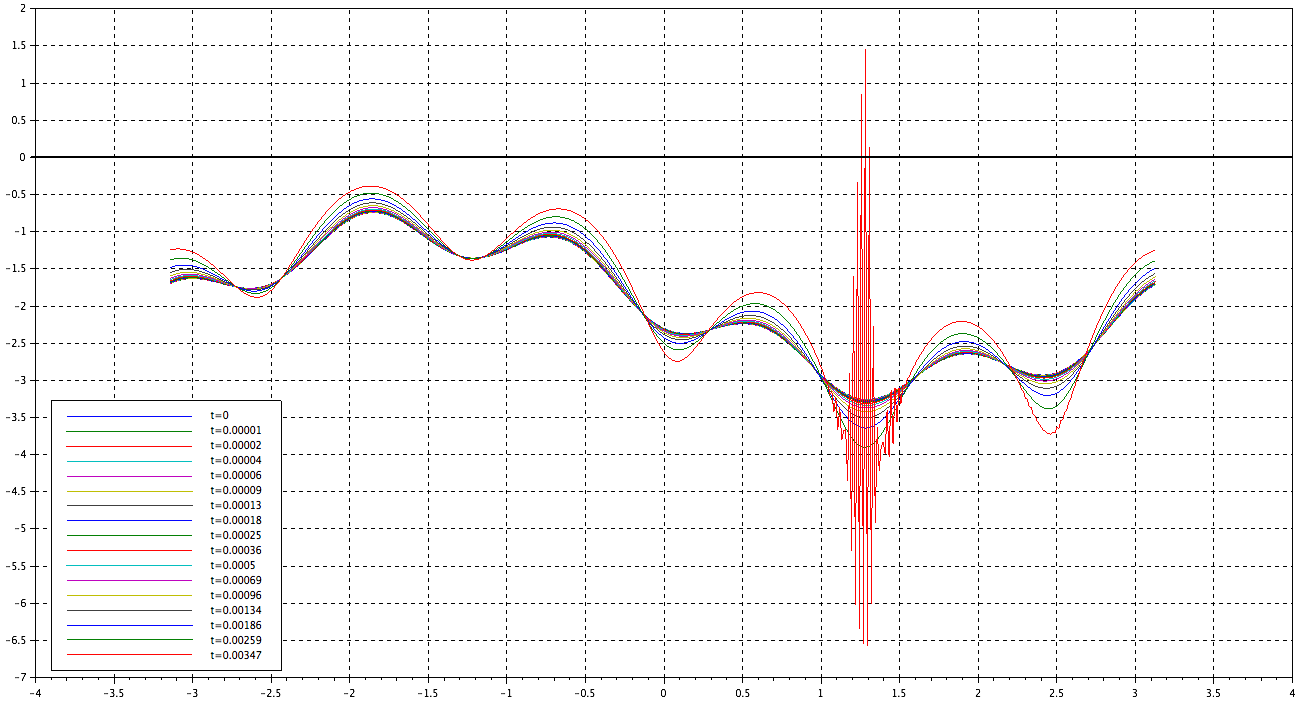}\\[1ex]
\includegraphics[width=.9\textwidth,height=.45\textheight]{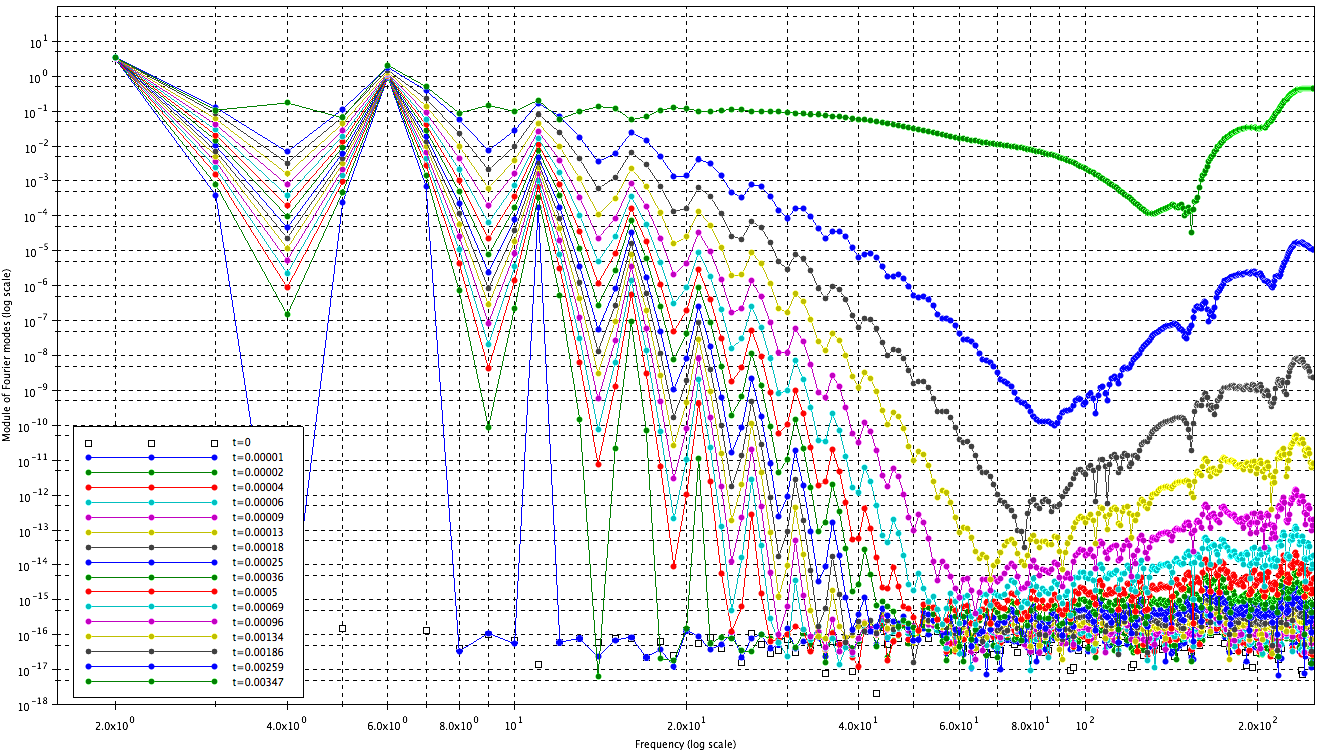}
\captionof{figure}{Evolution from negative data and evolution of the Fourier Spectrum showing the emergence of a singularity.}\label{SimD}
\end{center}
\end{minipage}

\subsection{Evolution from unsigned data with a positive average}

Let us write the Non-local Burgers equation~\eqref{e:main} in a form that separates the average value of the solution from its fluctuations around it. To this end, we introduce the zero-average function
$$v(t,x)=u(t,x)-\frac{1}{|\T^d|}\int_{\T^d} u(t,y) dy.$$
The average value can be recovered from $v$ by the momentum law (ML)
$$\int_{\T^d} u(t,y)dy = \int_{\T^d} u_0(x) dx + \int_0^t \|v(\tau)\|_{\dot{H}^{1/2}}^2 d\tau.$$
Note that the average of $u$ is globally bounded from above and below by:
$$\frac{1}{|\T^d|}\int_{\T^d} u_0(x)dx \leq  \frac{1}{|\T^d|}\int_{\T^d} u(t,y)dy \leq  \left(\frac{1}{|\T^d|}\int_{\T^d} u_0(x)^2 dx\right)^{1/2}.$$
The zero-average fluctuations $v$ evolve according to
\begin{equation}\label{eq:mean_value}
\partial_t v + \left(\frac{1}{|\T^d|}\int_{\T^d} u(t,y)dy\right) |\nabla|v = [v,|\nabla|]v  - \frac{1}{|\T^d|}\int_{\T^d} v |\nabla| v.
\end{equation}
Therefore, if the initial condition has a positive average, equation~\eqref{eq:mean_value} contains a uniform smoothing term on the left-hand side, while the right-hand side is a non-local, quadratic
but energy preserving non-linearity of order 1 i.e., formally:
$$\|v(t)\|_{L^2}^2 + 2 \int_0^t \left( \frac{1}{|\T^d|}\int_{\T^d} u(\tau,y)dy\right) \|v(\tau)\|_{\dot{H}^{1/2}} d\tau = \|v_0\|_{L^2}^2.$$

Figures \ref{SimB}, \ref{SimC1} and \ref{SimC2} show numerical simulations starting from different unsigned data but with a positive average.
 For the first one, the slight incursion in negative values does not seem to affect the solution.
The second one shows a localized transient singularity (which might have been smoothed out by the finite size of the mesh). The Fourier spectra at various time slices show
the emergence and the subsequent vanishing of the singularity.

Figure \ref{SimC2} displays an interesting feature of those transient instabilities: they appear to be localized along the negative local minima of the data. Their sequence of development seems also to start from the one with the highest absolute value and then proceed in order towards the ones of lower magnitude.

\bigskip
\par\noindent\textbf{Remark.}\quad
In its form \eqref{eq:mean_value}, the Non-local Burgers equation reveals some similarity with  the incompressible Navier-Stokes equation that can be written as
\[
\partial_t u - \nu \Delta u = - u \cdot \n u - \n p
\]
with $\operatorname{div} u=0$. We already pointed out the similarity of our model with the Euler part of NSE, especially if one computes the pressure in term of $u$. For the fluctuation $v$ we also consider that the
active ``viscosity" term on the left-hand side of \eqref{eq:mean_value} is similar to the kinematic viscosity in NSE. Despite this similarity with the notorious counterpart, we were able to prove global existence for our model. Of course, in our case this boils down to knowing extra structure, e.g.~the maximum principle. This once again highlights the importance of finding additional structural properties of the NSE in any attempt to approach its celebrated global regularity problem.

\bigskip\par\noindent
\begin{minipage}{\linewidth}
\begin{center}
\captionsetup{type=figure}
\includegraphics[width=.8\textwidth,height=.4\textwidth]{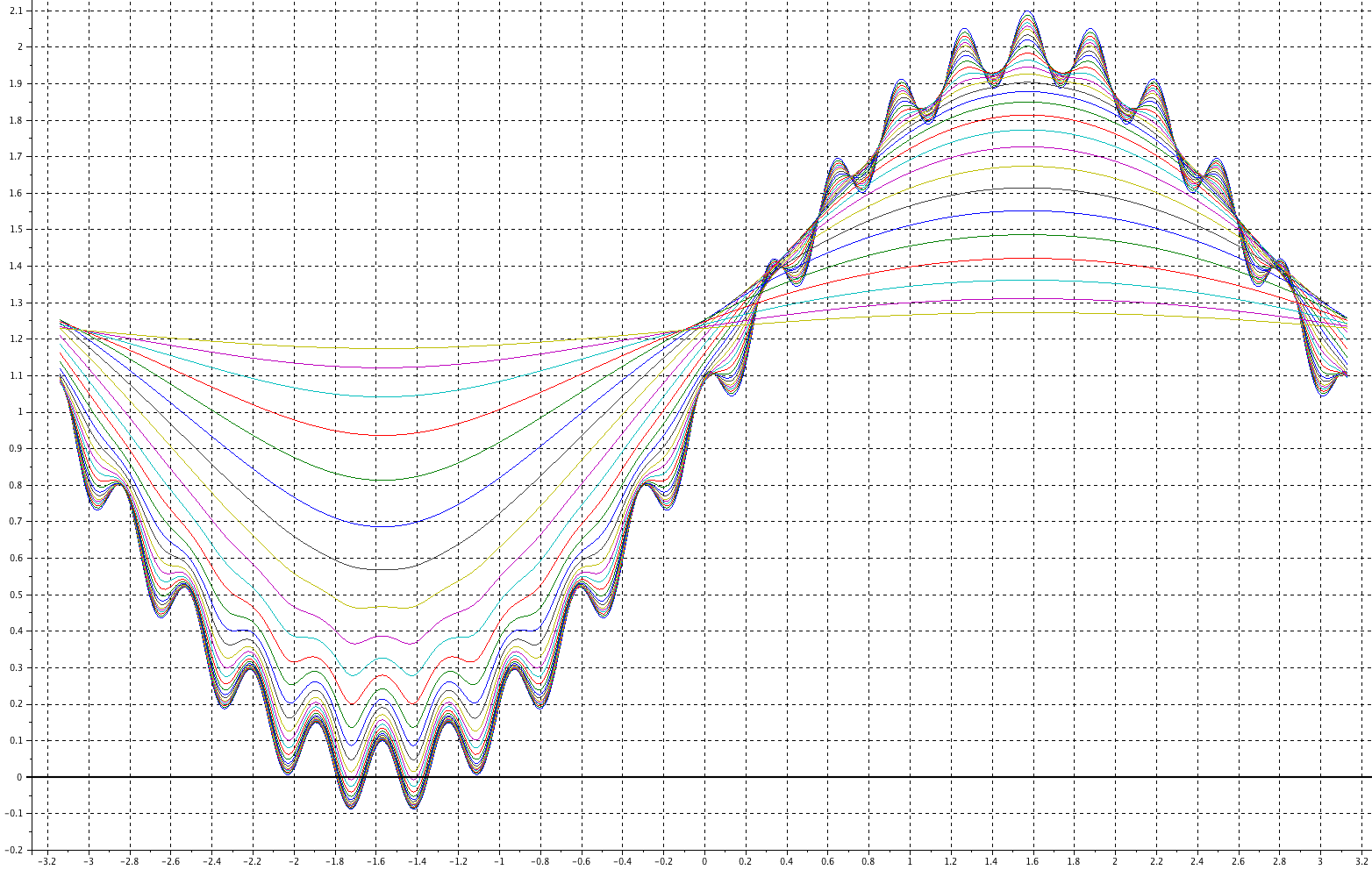}\\
\captionof{figure}{Evolution from $u_0=1+\sin(x)+\frac{1}{10}\cos(20x)$.}\label{SimB}

\bigskip
\includegraphics[width=.8\textwidth,height=.4\textwidth]{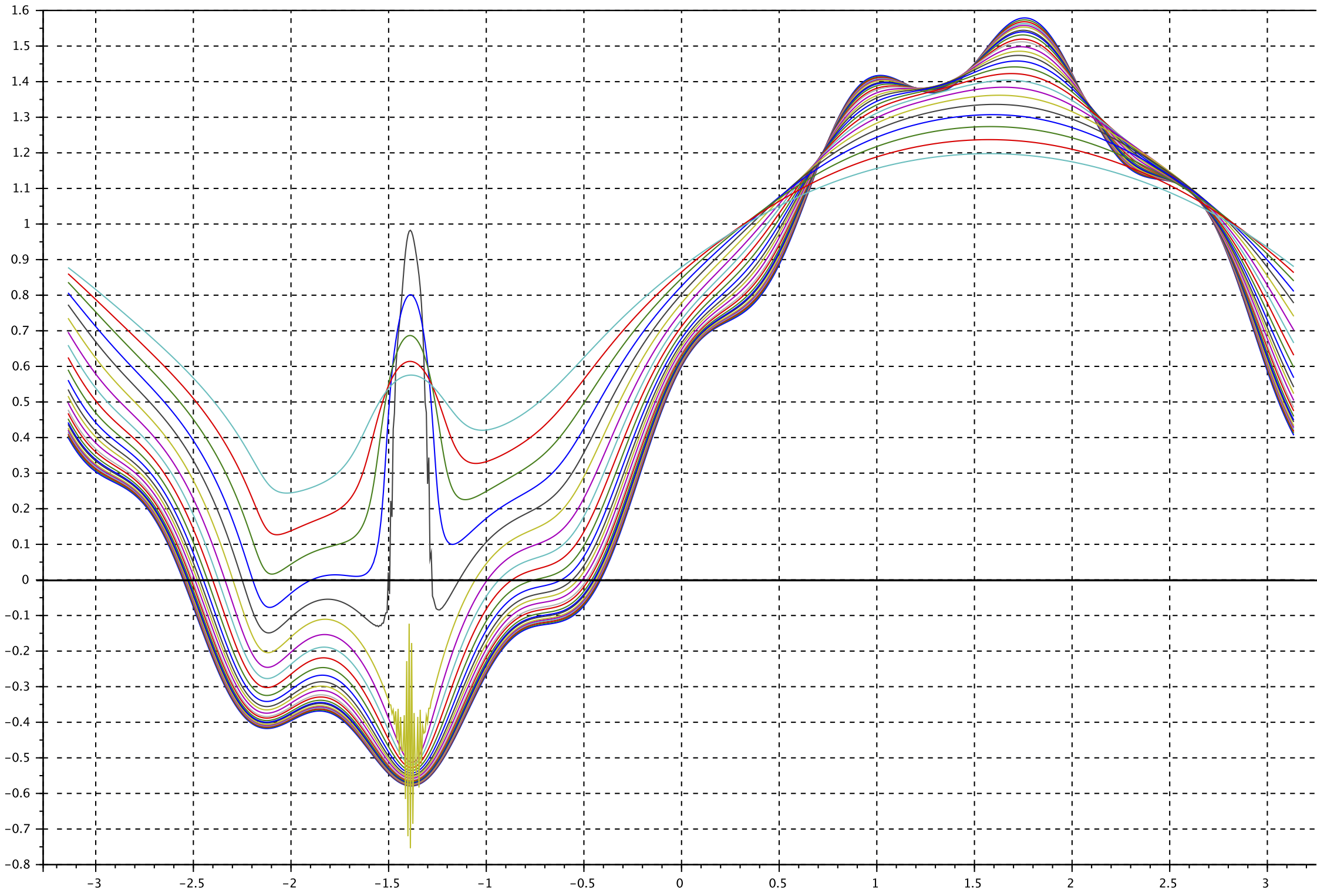}\\[1ex]
\includegraphics[width=.81\textwidth,height=.4\textwidth]{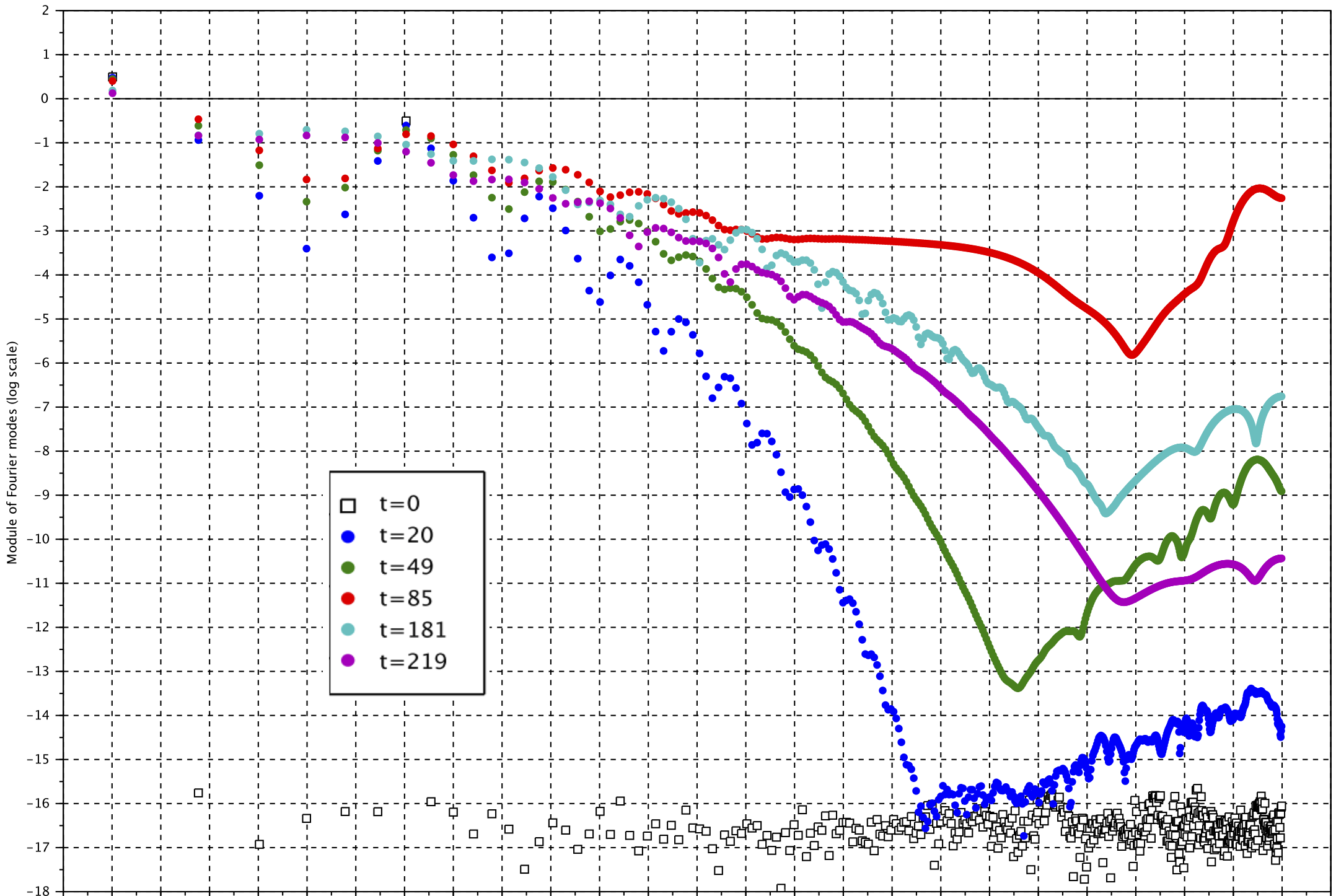}
\captionof{figure}{Evolution from $u_0=\frac{1}{2}+\sin(x)+\frac{1}{10}\cos(7x)$ and Fourier sprectrum.}\label{SimC1}
\end{center}
\end{minipage}

\begin{minipage}{\linewidth}
\begin{center}
\captionsetup{type=figure}
\includegraphics[width=.81\textwidth,height=.5\textwidth]{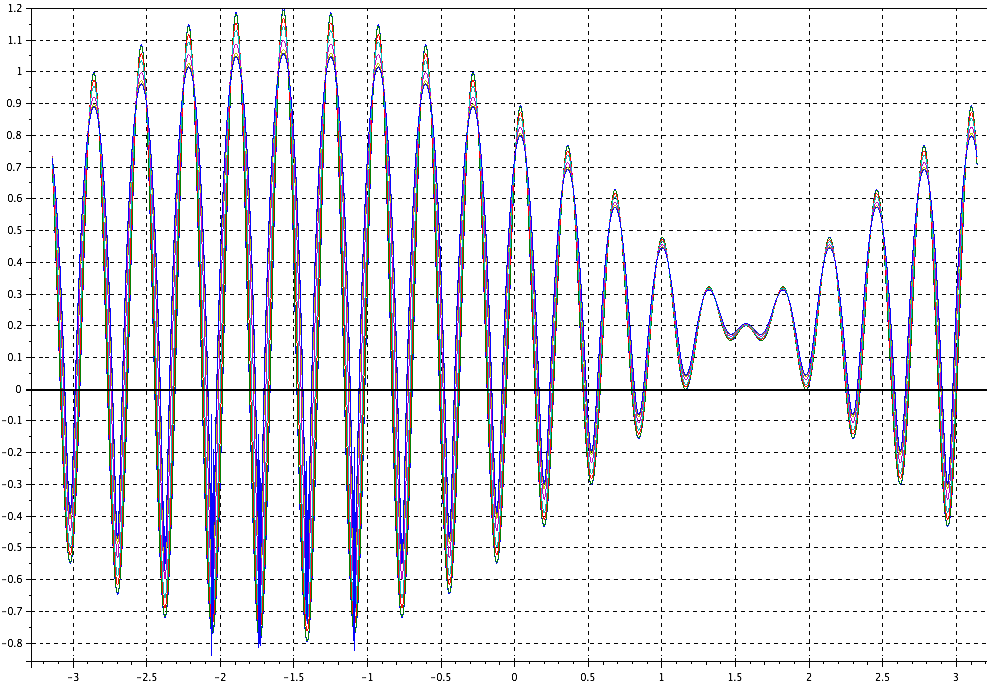}
\captionof{figure}{Evolution from $u_0=\frac{2}{10}+\frac{1}{2}\sin(19x)+\frac{1}{2}\cos(20x)$.}\label{SimC2}
\end{center}
\end{minipage}

\subsection{Global simulation with a finite element method}

To minimize the effects of time-discretization, we did numerical simulations of \eqref{e:main2} using finite elements, based on piecewise linear functions. The resulting problem is an ODE of the form
$A U' = J(U,U)$ where $U(t)\in \R^n$ represents the value of $(u(t,x_i))_{1\leq i\leq n}$, $A$ is a Toepliz matrix and $J$ is an $n$-dimensional vector of quadratic forms.
As energy is also preserved for the ODE (${}^t U AU=\text{Cte}$),  it is globally well posed for all times $t\in\R$ (forwards and backwards). 

As a happy coincidence due to the special structure of the non-linearity, diagonalizing the Toepliz matrix $A$ drastically simplifies each quadratic form in $J$, which then allows for a
direct computation of the solution.

Figure \ref{SimE} shows two solutions obtained in that manner with 20 grid-points. The interesting point of that simulation is that the flow
seems to connect a negative constant ground state to a positive one and that all the oscillations occur in the non sign-definite region. The level-sets of that region are shown above the 3D perspective.

\subsection{Perspective: the Frozen Non-local Burgers model}

As the previous numerical simulations have shown, the most interesting and mathematically challenging region is when~$u$ is non sign-definite. We feel that it is worth mentioning that the Non-local Burgers equation \eqref{e:main} comes with a natural twin in which the average value  is frozen instead of memorizing the $\dot{H}^{1/2}$ past norm. The Frozen Non-local Burgers equations reads:
\begin{equation}\label{eq:FNB}
\partial_t v = [v, |\nabla|] v - \frac{1}{|\T^d|}\int_{\T^d} v |\nabla| v.
\end{equation}
This equation should be compared with the previous equation \eqref{eq:mean_value} of the fluctuations $v$ of the solution of \eqref{e:main2}.
Formally, \eqref{eq:FNB} conserves momentum $\int_{\T^d} v(t,x)dx=\text{Cte}$. The previous intuition suggests that it would be most interesting to study it for vanishing average data, in which case it formally conserves the $L^2$-norm too.

\par\noindent
\begin{minipage}{\linewidth}
\begin{center}
\captionsetup{type=figure}
\includegraphics[width=.95\textwidth]{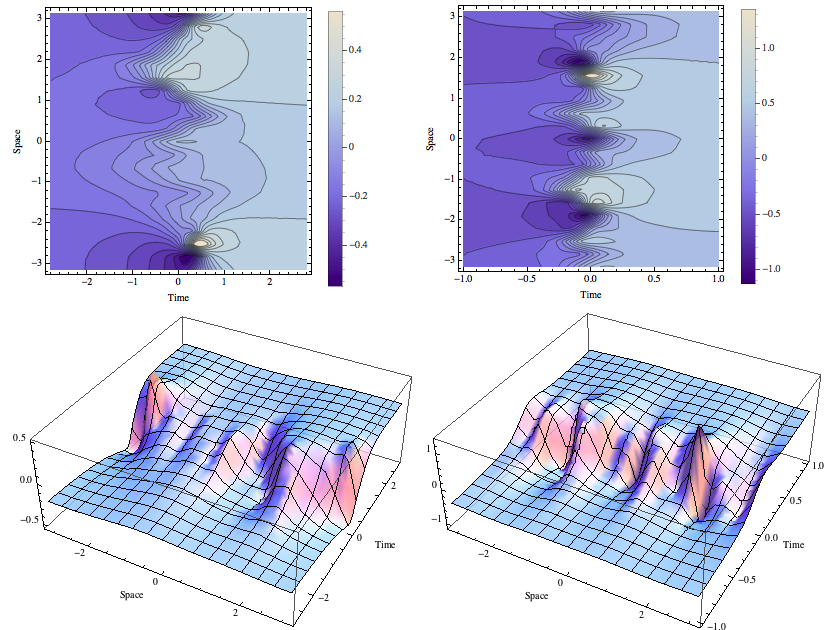}
\captionof{figure}{Global FEM simulations.}\label{SimE}
\end{center}
\end{minipage}

\bibliographystyle{plain}
\bibliography{reference}

\end{document}